\documentclass[leqno,12pt]{amsart} %leqno is the option to put formula numbers on the left side
\usepackage{amssymb}
\usepackage{amsmath}
\usepackage{amsthm}
\usepackage{amscd}
\usepackage{graphicx}
\usepackage[all]{xy}
\usepackage[dvips]{color}
\theoremstyle{plain} %text of this environment is typesetted in italics
\newtheorem{theorem}{\indent\sc Theorem}[section]
\newtheorem{lemma}[theorem]{\indent\sc Lemma}

\newtheorem{proposition}[theorem]{\indent\sc Proposition}

\newtheorem{conjecture}[theorem]{\indent\sc Conjecture}
\theoremstyle{definition} %text of this environment is typesetted in roman letters

\newtheorem{remark}[theorem]{\indent\sc Remark}
\newtheorem{example}[theorem]{\indent\sc Example}

\newtheorem{problem}[theorem]{\indent\sc Problem}
\newtheorem{question}[theorem]{\indent\sc Question}
\begin{document}

\title[Linearization of a semi-positive line bundle]
{Linearization of transition functions of a semi-positive line bundle along a certain submanifold} %title of paper and the running head option

\author[T. Koike]{Takayuki Koike} %first author's name and the running head option

%%%%%%%%%%%%%%% footnote %%%%%%%%%%%%%%%%
%In case \subjclass[2010] command is not effective
%(or the version of amsart.cls is old), write as follows instead:
%\renewcommand{\thefootnote}{\fnsymbol{footnote}}
%\footnote[0]{2010\textit{ Mathematics Subject Classification}.
%Primary 00; Secondary 00.}
%
%\thanks{ %acknowledgment of support etc. if any
%$^{*}$Thanks.
%}
%%%%%%%%%%%% Authors' addresses %%%%%%%%%%%%%
\address{% First Author
Department of Mathematics, Graduate School of Science, Osaka City University \endgraf
3-3-138, Sugimoto, Sumiyoshi-ku Osaka, 558-8585 \endgraf
Japan
}
\email{tkoike@sci.osaka-cu.ac.jp}
%%%%%%%%%%%%%%%%%%%%%%%%%%%%%%%%%%%%%%%%%

\maketitle

\begin{abstract}
Let $X$ be a complex manifold and $L$ be a holomorphic line bundle on $X$. 
Assume that $L$ is semi-positive, namely $L$ admits a smooth Hermitian metric with semi-positive Chern curvature. 
Let $Y$ be a compact K\"ahler submanifold of $X$ such that the restriction of $L$ to $Y$ is topologically trivial. 
We investigate the obstruction for $L$ to be unitary flat on a neighborhood of $Y$ in $X$. 
As an application, for example, we show %the following two results: 
the existence of nef, big, and non semi-positive line bundle on a non-singular projective surface. %, 
%and the existence of a Monge--Amp\`ere type foliation on a neighborhood of $Y$ 
%when $X$ is a surface which is not holomorphically convex under some technical conditions. 
%of which $Y$ is a leaf 
%when $X$ is a connected weakly $1$-complete K\"ahler surface 
%such that the compliment $X\setminus Y$ is a non holomorphically convex surface with analytically trivial canonical bundle if the normal bundle of $Y$ is topologically trivial. 
\end{abstract}

%%%%%%%%%%%%%%%%%%%%%%%%%%%%%%%%%%%%%%%%%%%
\section{Introduction}

Let $X$ be a complex manifold and $L$ be a holomorphic line bundle on $X$. 
Assume that $L$ is {\it semi-positive}, 
namely there exists a $C^\infty$ Hermitian metric $h$ such that $\sqrt{-1}\Theta_h$ is semi-positive at any point of $X$, where $\Theta_h$ is the Chern curvature tensor of $h$. 
Let $Y$ be a compact K\"ahler submanifold of $X$ such that the restriction $L|_Y$ of $L$ to $Y$ is topologically trivial. 
In this case, it is known that $L|_Y$ is unitary flat (see \S \ref{section:2_1}). 
Our interest is in the relation between the restriction of $L|_V$ of $L$ to a small tubular neighborhood $V$ of $Y$ and the unitary flat line bundle $\widetilde{L}$ on $V$ with $\widetilde{L}|_Y=L|_Y$ 
({\it flat extension}, the existence of such a line bundle $\widetilde{L}$ follows by considering the isomorphism 
$H^1(V, \mathrm{U}(1)) \to H^1(Y, \mathrm{U}(1))$, where we denote by $\mathrm{U}(1)$ the unitary group of degree $1$; i.e. $\mathrm{U}(1):=\{t\in\mathbb{C}\mid |t|=1\}$). 
As a tentative answer, let us pose the following: 

\begin{conjecture}\label{conj:main}
Let $X, L$, and $Y$ be as above. 
Then $L|_V$ is unitary flat for a sufficiently small neighborhood $V$ of $Y$ in $X$. 
\end{conjecture}

See also \cite[Conjecture 2.1]{K2019} and \cite[Theorem 1.1]{K3}. 
Towards solving this conjecture, 
we investigate the difference between $L$ and $\widetilde{L}$ in each finite order jet along $Y$ in the present paper. 
As applications, we show the following two results. 

\begin{theorem}\label{thm:main_1}
There exists a nef and big line bundle on a non-singular projective surface which is not semi-positive. 
\end{theorem}

\begin{theorem}\label{thm:main_2}
Let $X$ be a connected weakly $1$-complete K\"ahler manifold of dimension $2$ 
and $Y\subset X$ be a holomorphically embedded compact non-singular curve. 
Assume that the normal bundle $N_{Y/X}$ of $Y$ is topologically trivial, the canonical bundle $K_{X\setminus Y}$ of $X\setminus Y$ is holomorphically trivial, 
and that there exists a $C^\infty$ Hermitian metric $h$ on $[Y]$ with $\sqrt{-1}\Theta_h$ is semi-positive, where $[Y]$ is the holomorphic line bundle on $X$ which corresponds to the divisor $Y$.  
Then either the conditions $(i)$, $(ii)$ or $(iii)$ holds: \\
$(i)$ The surface $X\setminus Y$ is holomorphically convex,  \\
$(ii)$ There exists a non-singular holomorphic foliation $\mathcal{F}$ on a neighborhood of $Y$ such that $Y$ is a leaf of $\mathcal{F}$ and that $i_L^*\Theta_{h}\equiv 0$ holds on any leaf $L$ of $\mathcal{F}$, where $i_L\colon L\to X$ is the inclusion, or\\
$(iii)$ It hold that $\Theta_h\wedge \Theta_h\equiv 0$ and that the function $\rho\colon X\to \mathbb{R}_{>0}$ which maps a point $p\in X$ to the trace of $\Theta_h|_p$ with respect to some Hermitian metric on $X$ is flat at any point of $Y$: i.e. $\rho(p) = o(({\rm dist}\,(p, Y))^n)$ for any positive integer $n$ as $p$ approaches to $Y$, where ``{\rm dist}" is a local Euclidean distance. 
\end{theorem}

Note that, in the proof of Theorem \ref{thm:main_1}, we prove that the line bundle $L$ is not semi-positive for a variant of Grauert's example $(X, L)$ of a nef and big line bundle $L$ on a non-singular projective surface $X$, see \cite[Problem 2.2]{FT} and Example \ref{ex:nefbig}. 
Note also that the existence of such line bundles for higher dimensional projective manifolds has been known (\cite[Example 2.14]{D} and \cite[Example 5.4]{BEGZ}, see also Remark \ref{rmk:begz}). 
The motivation of Theorem \ref{thm:main_2} comes from the study of a neighborhood of an elliptic curve embedded in a K\"ahler surface $S$ such that it is a suitable irreducible component of the divisor corresponding to the canonical bundle $K_S$ of $S$ when $K_S$ or the anti-canonical bundle $K_S^{-1}$ is semi-positive (see Question \ref{q:surface_main}). 
In the proof of Theorem \ref{thm:main_2}, \cite[Proposition 2]{B} plays an important role. 

Our idea to compare the line bundle $L$ with the flat extension $\widetilde{L}$ of $L|_Y$ comes from Ueda theory on the classification of the analytic structures of a neighborhood of a submanifold (\cite{U}, see also \cite{N} or \S \ref{section:2_2} here). 
Let $Y$ be a compact non-singular curve holomorphically embedded in a non-singular surface $X$ such that the normal bundle $N_{Y/X}$ is topologically trivial. 
We denote by $[Y]$ the line bundle on $X$ which corresponds to the divisor $Y$ (or the invertible sheaf $\mathcal{O}_X(Y)$). 
Let $\widetilde{N}$ be the flat extension of $N_{Y/X}$, namely $\widetilde{N}$ is the unitary flat line bundle on a tubular neighborhood of $Y$ such that $\widetilde{N}|_Y=N_{Y/X}$. 
As $[Y]|_Y$ also coincides with $N_{Y/X}$, one can consider the difference between $[Y]$ and $\widetilde{N}$ in the first order jet along $Y$, by which the first Ueda's obstruction class 
$u_1(Y, X)\in H^1(Y, N_{Y/X}^{-1})$ is defined. 
When $u_1(Y, X)=0$, or equivalently when $[Y]$ and $\widetilde{N}$ coincide in the first order jet, 
one can define the second Ueda's obstruction class 
$u_2(Y, X)\in H^1(Y, N_{Y/X}^{-2})$ by comparing $[Y]$ and $\widetilde{N}$ in the second order jet along $Y$. 
In the case where all (similarly and inductively defined) Ueda's obstruction classes $u_n(Y, X)$'s vanish, Ueda gave a sufficient condition for the coincidence of $[Y]$ and $\widetilde{N}$ on a neighborhood of $Y$ in $X$ \cite[Theorem 3]{U}. 
As the coincidence of $[Y]$ and $\widetilde{N}$ can be interpreted as the vertical linearizability of a neighborhood of $Y$ (or, more precisely, the linearizability of the transition functions of the system of local defining functions of $Y$), this Ueda's theorem can be regarded as a generalization of Arnold's linearization theorem \cite{A} of a neighborhood of an elliptic curve. 
Note that the definition of Ueda's obstruction classes and this type of vertical linearization theorems can naturally be generalized into the cases of general dimensions if the normal bundle is unitary flat, see \cite{K2018} or \S \ref{section:2_2} here. 

Again, let $Y$ be a compact non-singular curve holomorphically embedded in a non-singular surface $X$ such that the normal bundle $N_{Y/X}$ is topologically trivial. 
When $u_n(Y, X)$ is a non-zero element of $H^1(Y, N_{Y/X}^{-n})$ for some positive integer $n$, the pair $(Y, X)$ is said to be {\it of finite type}. 
Ueda investigated the details of the complex analytical properties of a neighborhood of $Y$ also in this case \cite[Theorem 1, 2]{U}. 
By applying one of these results of Ueda, the author showed the non semi-positivity of the line bundle $[Y]$ when the pair $(Y, X)$ is of finite type \cite[Theorem 1.1]{K3}, which is one of the biggest motivation of the present paper since it can be regarded as a partial answer to Conjecture \ref{conj:main}. 

Let $X$ be a complex manifold, $Y\subset X$ be a compact K\"ahler submanifold, and $L$ be a line bundle on $X$ such that the restriction $L|_Y$ is topologically trivial. 
In \S \ref{section:3}, according to the spirit of Ueda's classification, we pose an obstruction class 
$u_1(Y, X, L) \in H^1(Y, N_{Y/X}^*)$ by comparing $L$ and the flat extension $\widetilde{L}$ of $L|_Y$ in the first order jet along $Y$ (so that it vanishes if these two line bundles coincide in the first order jet), where $N_{Y/X}^*$ is the dual vector bundle of the normal bundle. 
Note that, when $Y$ is a hypersurface with topologically trivial normal bundle, the definition of two obstruction classes $u_1(Y, X)$ and $u_1(Y, X, L)$ coincide if $L=[Y]$. 
By using this first obstruction class, we show the following: 

\begin{theorem}\label{thm:1}
Let $X$ be a complex manifold, $Y\subset X$ be a compact K\"ahler submanifold, and $L$ be a line bundle on $X$ such that the restriction $L|_Y$ is topologically trivial. 
Assume that $u_1(Y, X, L)\not=0$. 
Then $L$ is not semi-positive. 
\end{theorem}

\begin{theorem}\label{thm:main_linearization}
Let $X$ be a non-singular surface, 
$L$ a holomorphic line bundle on $X$, 
and $Y$ be a non-singular compact curve holomorphically embedded into $X$ such that ${\rm deg}\,N_{Y/X}\leq \min\{-1,\ 2-2g\}$, where $g$ is the genus of $Y$. 
Assume that ${\rm deg}\,L|_Y=0$ and that $L\otimes [Y]^{-m}$ is semi-positive for some positive integer $m$. 
Then $L$ is semi-positive if and only if $u_1(Y, X, L)=0$. 
\end{theorem}

Note that Theorem \ref{thm:main_1} follows from Theorem \ref{thm:1} and a concrete calculation of the obstruction class $u_1(Y, X, L)$ for a variant of Grauert's example (Example \ref{ex:nefbig}, see \S \ref{section:4_1}). 
Note also that Theorem \ref{thm:main_linearization} can be regarded as a generalization of the semi-positivity result for the case of ${\rm deg}\,N_{Y/X}\leq \min\{0,\ 4-4g\}$ mentioned just after \cite[Theorem 1.2]{K2}, since $u_1(Y, X, L)=0$ automatically follows from Kodaira vanishing theorem when ${\rm deg}\,N_{Y/X}< 2-2g$. 

Though it seems to be natural to consider the second obstruction class when $u_1(Y, X, L)$ vanishes, 
there is a difficulty in general on the well-definedness. 
More precisely, though one can define a class of the first cohomology group $H^1(Y, S^2N_{Y/X}^*)$ of the second symmetric tensor product bundle $S^2N_{Y/X}^*$ of $N_{Y/X}^*$ by comparing the difference between $L$ and $\widetilde{L}$ in second order jet, 
this class may depend on the choice of a system of local defining functions of $Y$ and frames of $L$. 
The same type of well-definedness problem also occurs in more higher order jets. 
In \S \ref{section:3_2}, we give a sufficient condition for the obstruction class 
$u_n(Y, X, L) \in H^1(Y, S^nN_{Y/X}^*)$ we will inductively define to be well-defined (Proposition \ref{prop:well-def_main}). 
Especially we observe the well-definedness of the obstruction classes and 
give a necessary condition for $L$ to be semi-positive by using higher obstruction classes 
when $Y$ is a hypersurface and $N_{Y/X}^*$ is either topologically trivial or not pseudo-effective in \S \ref{section:3_4} and \ref{section:3_5}. 
For example, we have the following for the case where $N_{Y/X}$ is topologically trivial. 

\begin{theorem}\label{thm:flat_u2}
Let $X$ be a complex manifold and $Y$ be a non-singular compact hypersurface of $X$ which is K\"ahler. 
Assume that the normal bundle $N_{Y/X}$ is topologically trivial, 
and that $[Y]$ is semi-positive. 
Then it holds that $u_1(Y, X)=u_2(Y, X)=0$. 
\end{theorem}

The organization of the paper is as follows. 
In \S 2, we will explain some previous or known results and fundamental facts on Hermitian metrics on line bundles and Ueda theory. 
In \S 3, we define the obstruction classes $u_n(Y, X, L)$'s and investigate some properties of them especially when $L$ is semi-positive. 
In \S 4, we apply some results in \S 3 to show Theorems \ref{thm:main_1}, \ref{thm:main_2}, and \ref{thm:main_linearization}. 
In \S 5, we make some discussion and pose some problems on our obstruction classes, neighborhoods of a submanifold, and the semi-positivity of the line bundles. 
\vskip3mm
{\bf Acknowledgment. } 
The author would like to give heartfelt thanks to Professors
Shigeharu Takayama and Valentino Tosatti for discussions on \cite[Problem 2.2]{FT}, 
which gave him one of the biggest motivations of this study. 
He is also grateful to Professor Takeo Ohsawa for his comments and suggestions of inestimable value, 
since the main part of this study was done during their stay in Wuppertal at the beginning of 2020 and the author believe that it was a great inspiration for this work. 
Thanks are also due to Doctor Nicholas M${}^{\rm c}$Cleerey for his helpful comments and warm encouragements, 
and to Professor Dano Kim for his kindly giving him information on his work on the existence of nef, big and not semi-positive line bundles for higher dimensional projective manifolds. 
This work was supported by JSPS Grant-in-Aid for Research Activity Start-up 
18H05834, 
by MEXT Grant-in-Aid for Leading Initiative for Excellent Young Researchers（LEADER) No. J171000201,  
and partly by Osaka City University Advanced Mathematical Institute (MEXT Joint Usage/Research Center on Mathematics and Theoretical Physics).
%%%%%%%%%%%%%%%%%%%%%%%%%%%%%%%%%%%%%%%%%%%

%%%%%%%%%%%%%%%%%%%%%%%%%%%%%%%%%%%%%%%%%%%
\section{Preliminaries}\label{section:2}

%%%%%%%%%%%%
\subsection{Curvature semi-positivity of Hermitian metrics on line bundles}\label{section:2_1}

Let $X$ be a complex manifold 
and $L$ be a holomorphic line bundle on $X$. 
For a positive integer $m$, we denote by $L^m$ the $m$-th tensor power $L^{\otimes m}$ and by $L^{-m}$ the $m$-th tensor power $(L^*)^{\otimes m}$ of the dual bundle $L^*$ of $L$. 

Let $\{V_j\}$ be an open covering of $X$. 
When $\{V_j\}$ is sufficiently fine, one can take a system of local frames $\{(V_j, e_j)\}$ of $L$; 
i.e. each $e_j$ is a nowhere vanishing holomorphic section of $L$ on $V_j$. 
The ratio $s_{jk}:=e_j/e_k$, which is a nowhere vanishing holomorphic function on $V_{jk}:=V_j\cap V_k$, can be regarded as the transition function of $L$. 
Let $h$ be a $C^\infty$ Hermitian metric on $L$. 
The function $\varphi_j:=-\log |e_j|_h^2$ on each $V_j$ is called as the {\it local weight function} of $h$. 
It follows from a simple calculation that $\{(V_j, \partial\overline{\partial}\varphi_j)\}$ glue to each other to define a global $(1, 1)$-form $\Theta_h$, 
which coincides with the Chern curvature tensor of $h$. 
Therefore it follows that $h$ is semi-positively curved if and only if any local weight function $\varphi_j$ is plurisubharmonic. 

We say that $L$ is {\it unitary flat} if one can choose a system of local frames $\{(V_j, e_j)\}$ of $L$ such that all the transition functions $s_{jk}$'s are elements of $\mathrm{U}(1)$ (by taking a refinement of $\{V_j\}$ if necessary). 
It is known that $L$ is unitary flat if $X$ is a compact K\"ahler manifold and $L$ is topologically trivial (Kashiwara's theorem, see \cite[\S 1]{U}). 
Assume that $X$ is compact, $L$ is unitary flat, and that a system $\{(V_j, e_j)\}$ satisfies that all the transition functions $s_{jk}$'s are elements of $\mathrm{U}(1)$. 
Then, for any Hermitian metric $h$ on $L$, it clearly holds that $\{(V_j, -\log |e_j|_h^2)\}$ glue to each other to define a global function on $X$. When $h$ is semi-positively curved, this function must be a constant function by the maximal principle. 

Let $D$ be a compact hypersurface of $X$. 
Then one can define a corresponding holomorphic line bundle $[D]$ on $X$ so that the sheaf $\mathcal{O}_X([D])$ of holomorphic sections of $[D]$ is isomorphic to the invertible sheaf $\mathcal{O}_X(D)$ of rational functions on $X$ which may have a pole only along $D$ with degree at most one. 
For an open (Stein) covering $\{V_j\}$ of $X$, one can take a system of local frames $\{(V_j, e_j)\}$ of $[D]$ such that $e_j$ corresponds, via the isomorphism $\mathcal{O}_X([D])\cong \mathcal{O}_X(D)$, to the constant function $1$ if $V_j\cap D=\emptyset$ and to the meromorphic function $1/w_j$ if $V_j\cap D\not=\emptyset$ for some local holomorphic defining function $w_j$ of $Y$ in $V_j$. 
In this case, $\{(V_j, w_j\cdot e_j)\}$ patches to each other to define a global section $f_D\colon X\to [D]$, which is called as the {\it canonical section} of $[D]$. 
Assume that $[D]$ admits a $C^\infty$ Hermitian metric $h$. 
Let $\phi_D$ be a locally $L^1$ function on $X$ defined by 
$\phi_D := -\log |f_D|_h^2$. 
On each $V_j$, one clearly have that $\phi_D = \varphi_j$ if $V_j\cap D=\emptyset$ and 
$\phi_D = -\log |w_j|^2 + \varphi_j$ if $V_j\cap D\not=\emptyset$. 
Thus it follows that $\phi_D$ is a plurisubharmonic function of the compliment $X\setminus D$ in this case. 
By considering this function $\phi_D$, problems on the existence of a metric with semi-positive curvature on $[D]$ can be reworded to problems on the complex analytical convexity of the complement $X\setminus D$ (see the arguments in the proof of \cite[Theorem 1.1]{K3} or \cite[\S 2.1]{K2019}). 
Note that, even when $D$ is smooth and $N_{D/X}$ is unitary flat, it may possible that $[D]$ is not semi-positive. 
Indeed, \cite[Example 1.7]{DPS} gives such an example. 
Note also that $[D]$ is semi-positive if there exists a neighborhood $V$ of $D$ in $X$ such that $[D]|_V$ is unitary flat. 
This can be shown by using ``regularized minimum construction", see \cite{K2}, \cite[\S 5]{K2018}, or \cite[\S 2.1]{K2019} for the detail. 

%%%%%%%%%%%%

%%%%%%%%%%%%
\subsection{Ueda Theory}\label{section:2_2}

Let $X$ be a complex manifold and $Y\subset X$ be a holomorphically embedded compact complex submanifold with unitary flat normal bundle. 

In \cite{U}, Ueda investigated the complex analytic structure on a neighborhood of $Y$ when $X$ is a surface and $Y$ be a curve by defining obstruction classes as we shortly explained in \S 1 (see also \cite{N}). 
In \cite{K2018}, we investigated a higher codimensional analogue of Ueda theory. 
In this subsection, we will summarize our notions on this generalized version of Ueda theory for reader's convenience, see also \cite[\S 2.2]{K2019}. 

Let $X$ be a complex manifold and $Y\subset X$ be a compact complex submanifold of codimension $r\geq 1$ such that $N_{Y/X}$ is unitary flat. 
Take a finite open covering $\{U_j\}$ of $Y$ and a neighborhood $V_j$ of $U_j$ in $X$ and 
a defining function $w_j\colon V_j\to \mathbb{C}^r$ of $U_j$ for each $j$: 
i.e. $w_j$ is a holomorphic function on $V_j$ such that ${\rm div}(w_j^{\lambda})$'s transversally intersect along $U_j$, where $w_j^{\lambda}\colon V_j\to\mathbb{C}$ is the composition of $w_j$ and $\lambda$-th projection map $\mathbb{C}^r\to\mathbb{C}$. 
By a simple argument, one may assume that $dw_j=S_{jk}dw_k$ holds on each $U_{jk}:=U_j\cap U_k$ for some unitary matrix $S_{jk}\in\mathrm{U}(r)$ by changing $w_j$'s if necessary, where 
\[
dw_j:=\left(
    \begin{array}{c}
      dw_j^{1} \\
      dw_j^{2} \\
      \vdots \\
      dw_j^{r}
    \end{array}
  \right).  
\]
We call such a system $\{(V_j, w_j)\}$ of local defining functions of $Y$ as a system of {\it type $1$}. 
By shrinking $V_j$'s if necessary again, we assume that, for each $j$, there exists a holomorphic surjection ${\rm Pr}_{U_j}\colon V_j\to U_j$ such that $(w_j, z_j\circ {\rm Pr}_{U_j})$ are coordinates of $V_j$, where $z_j$ is a coordinate of $U_j$. 
In what follows, for any holomorphic function $f$ on $U_j$, we denote by the same letter $f$ the pull-buck ${\rm Pr}_{U_j}^*f:=f\circ {\rm Pr}_{U_j}$. 
On $U_j$ and $U_k$ such that $U_{jk}\not=\emptyset$, 
one have the series expansion 
\[
S_{jk}\cdot \left(
    \begin{array}{c}
      w_k^{1} \\
      w_k^{2} \\
      \vdots \\
      w_k^{r}
    \end{array}
  \right)
=
\left(
    \begin{array}{c}
      w_j^{1} \\
      w_j^{2} \\
      \vdots \\
      w_j^{r}
    \end{array}
  \right)
+\sum_{|a|\geq 2} \left(
    \begin{array}{c}
      f_{kj, a}^{(1)}(z_j) \\
      f_{kj, a}^{(2)}(z_j) \\
      \vdots \\
      f_{kj, a}^{(r)}(z_j)
    \end{array}  \right)
\cdot w_j^a, 
\]
where $a=(a_1, a_2, \dots, a_r)$ is the multiple index running all the elements of $(\mathbb{Z}_{\geq 0})^r$ with $|a|:=\sum_{\lambda=1}^ra_\lambda$ is larger than or equal to $2$, 
$f_{kj, a}^{(\lambda)}$'s are holomorphic functions on $U_{jk}$ (we regard this also as a function defined by $({\rm Pr}_{U_j}|_{{\rm Pr}_{U_j}^{-1}(U_{jk})})^*f_{kj, a}^{(\lambda)}$), 
and $w_j^a:=\prod_{\lambda=1}^r(w_j^{\lambda})^{a_\lambda}$. 
For a positive integer $m$, we say that the system $\{(V_j, w_j)\}$ of local defining functions is of {\it type $m$} if $f_{kj, a}\equiv 0$ holds for any $a$ with $|a|\leq m$ and any $j, k$ with $U_{jk}\not=\emptyset$. 
If $\{(V_j, w_j)\}$ is of type $m$, it follows that 
\[
\left\{\left(U_{jk}, \ \sum_{\lambda=1}^r\sum_{|a|=m+1}f_{kj, a}^{(\lambda)} \, \frac{\partial}{\partial w_j^{(\lambda)}}\otimes dw_j^a \right)\right\}
\]
satisfies the $1$-cocycle condition, 
and thus it defines an element of $H^1(Y, N_{Y/X}\otimes S^{m+1}N_{Y/X}^*)$. 
We denote this cohomology class by $u_m(Y, X)$, which is the definition of the $m$-th Ueda class. 
Ueda classes $u_m(Y, X)$ is well-defined up to the $\mathrm{U}(r)$-action of $H^1(Y, N_{Y/X}\otimes S^{m+1}N_{Y/X}^*)$, namely $[u_m(Y, X)]\in H^1(Y, N_{Y/X}\otimes S^{m+1}N_{Y/X}^*)/\mathrm{U}(r)$ does not depend on the choice of the system of type $m$. 

From a simple observation, one have that there exists a system of type $m+1$ if and only if $u_m(Y, X)=0$, in which case one can also define $u_{m+1}(Y, X)$. 
The pair $(Y, X)$ is said to be of {\it finite type} if, for some positive integer $n$, there exists a system of type $n$ such that $u_n(Y, X)\not=0$. 
Otherwise, the pair is said to be of {\it infinite type}. 
%%%%%%%%%%%%

%%%%%%%%%%%%
\subsection{A fundamental lemma}\label{section:2_3}

Though the following lemma is fundamental, it plays an important role in the linearizing procedure we will see in next section. 

\begin{lemma}\label{lem_positive_function}
Let $\Omega$ be a neighborhood of the origin in the complex plane $\mathbb{C}$ with the standard coordinate $w$, 
and $\Phi\colon \Omega\to \mathbb{R}$ be a function. 
Assume that $\Phi(w)\geq 0$ for any $w\in \Omega$ and that, for a positive integer $n$, $\Phi$ satisfies 
\[
\Phi(w) = \sum_{p=0}^n c_p\cdot w^p\overline{w}^{n-p} + O(|w|^{n+1})
\]
as $|w|\to 0$, where $c_p$'s are complex constants. 
Then the following holds: \\
$(i)$ When $n$ is odd, $c_p= 0$ for any $p\in \{0, 1, 2, \dots, n\}$. \\
$(ii)$ When $n$ is even, the constant $c_{n/2}$ is a non-negative real number. Any of the other constants are zero if $c_{n/2}=0$. 
\end{lemma}

\begin{proof}
We let $\Psi(w):=\sum_{p=0}^n c_p\cdot w^p\overline{w}^{n-p}$. 
As $\Phi$ is real valued, one have that $\Phi=\overline{\Phi}$, 
which implies that $\Psi$ is also a real-valued function. 
By considering a function $f_{w}\colon \mathbb{R}_{>0}\to \mathbb{R}_{\geq 0}$ defined by 
$f_{w}(r) := \Phi(rw) = r^n\cdot \Psi(w) +O(r^{n+1})$ for each $w\in\mathbb{C}$, 
one have that $\Psi(w)$ is non-negative at any point in $\mathbb{C}$. 

Denote by $\Delta$ the unit disc $\{w\in \mathbb{C}\mid |w|<1\}$ of $\mathbb{C}$. 
Then one have that 
\begin{align*}
\int_{\Delta}\Psi(w)\,d\lambda
&=\sum_{p=0}^n c_p\int_{\Delta} w^p\overline{w}^{n-p}\,d\lambda \\
&=\sum_{p=0}^n c_p\cdot \left(\int_0^1\ell^n\cdot \ell\,d\ell\right)\cdot \int_0^{2\pi} e^{\sqrt{-1}(2p-n)\theta}\,d\theta \\
&= \begin{cases}
0 & \text{if}\ n\ \text{is}\ \text{odd}\\
\frac{2\pi}{n+2}\cdot c_{n/2} & \text{if}\ n\ \text{is}\ \text{even}, 
\end{cases}
\end{align*}
where we denote by $d\lambda$ the Lebesgue measure, 
from which the assertions follow. 
\end{proof}

%%%%%%%%%%%%

%%%%%%%%%%%%%%%%%%%%%%%%%%%%%%%%%%%%%%%%%%%

%%%%%%%%%%%%%%%%%%%%%%%%%%%%%%%%%%%%%%%%%%%
\section{Linearization of a metric with semi-positive curvature}\label{section:3}

%%%%%%%%%%%%
\subsection{Set-up and notation}\label{section:3_1}
Let $X$ be a complex manifold with ${\rm dim}\,X=d+r$, 
$Y\subset X$ a compact K\"ahler submanifold of ${\rm dim}\,Y=d$ ($d, r\geq 1$), 
and $L\to X$ be a holomorphic line bundle such that the restriction $L|_Y$ is topologically trivial. 
Note that, in what follows, the assumption that $Y$ is compact K\"ahler is only needed to assure that 
$L|_Y$ is unitary flat, and that any unitary flat line bundle on $Y$ is trivial as a unitary flat line bundle if it is analytically trivial (see \S \ref{section:2_1}). 

Take a sufficiently fine finite open covering $\{U_j\}$ of $Y$ 
and a sufficiently small Stein open subset $V_j$ of $X$ such that $V_j\cap Y=U_j$. 
Denote by $V$ the neighborhood $\bigcup_jV_j$ of $Y$. 
Let $w_j \colon V_j\to \mathbb{C}^r$ ($w_j=(w_j^1, w_j^2, \dots, w_j^r)$) be local defining functions of $Y$, 
and 
$z_j=(z_j^1, z_j^2, \dots, z_j^d)$ be coordinates of $U_j$. 
By shrinking $V_j$'s if necessary, we assume that there exists a holomorphic surjection ${\rm Pr}_{U_j}\colon V_j\to U_j$ such that $(z_j^1\circ {\rm Pr}_{U_j}, z_j^2\circ {\rm Pr}_{U_j}, \dots, z_j^d\circ {\rm Pr}_{U_j}, w_j)$ are coordinates of $V_j$. 
In what follows, for any holomorphic function $f$ on $U_j$, we denote by the same letter $f$ the pull-buck ${\rm Pr}_{U_j}^*f:=f\circ {\rm Pr}_{U_j}$ for each $j$. 
We also simply denote by $(z_j, w_j)$ our local coordinates on $V_j$. 
Take a local frame $e_j$ of $L$ on each $V_j$. 
As $L|_Y$ admits a structure as unitary flat line bundle, one can take $e_j$'s such that $t_{jk}^{-1}\cdot e_k|_{U_{jk}}=e_j|_{U_{jk}}$ holds for some $t_{jk}\in\mathrm{U}(1)$ on each $U_{jk}:=U_j\cap U_k$. 
Then the expansion of the ratio function $t_{jk}^{-1}e_k/e_j$ is in the form 
\[
\frac{t_{jk}^{-1}e_k}{e_j} = 1+\sum_{|\alpha|\geq 1} f_{kj, \alpha}(z_j)\cdot w_j^\alpha. 
\]
%where $\alpha=(\alpha_1, \alpha_2, \dots, \alpha_r)\in(\mathbb{Z}_{\geq 0})^r$ is a multi-index and $w_j^\alpha := \textstyle\prod_{\lambda=1}^r (w_j^\lambda)^{\alpha_\lambda}$. 
We say that the system $\{(V_j, e_j, w_j)\}$ is of {\it type $n$} if 
$f_{kj, \alpha}\equiv0$ for any $j, k$ and any $\alpha$ with $|\alpha| < n$. 

Assume that $L$ is semi-positive. 
Take a $C^\infty$ Hermitian metric $h$ on $L$ with semi-positive curvature. 
Denote by $\varphi_j(z_j, w_j)$ the local weight function of $h$ on each $V_j$. 
Note that $\varphi_j$'s are plurisubharmonic. 
As $|t_{jk}|=1$, one have that 
\[
\varphi_k(z_k, w_k)
 = -\log |e_k|_{h(z_k, w_k)}^2
 = -\log |e_j|_{h(z_j, w_j)}^2-\log \left|1+\sum_{|\alpha|\geq 1} f_{kj, \alpha}(z_j)\cdot w_j^\alpha\right|^2
\]
if $(z_j, w_j)=(z_k, w_k)$. 
By considering Taylor expansion of the function $x\mapsto \log (1+x)$, one have that 
\begin{equation}\label{eq:diff_phi}
\varphi_k-\varphi_j = 
-\sum_{|\alpha|=n} f_{kj, \alpha}(z_j)\cdot w_j^\alpha
-\sum_{|\alpha|=n} \overline{f_{kj, \alpha}(z_j)}\cdot \overline{w_j^\alpha}
+O(|w_j|^{n+1})
\end{equation}
holds if $\{(V_j, e_j, w_j)\}$ is a system of type $n$, 
where 
$|w_j|:=\sqrt{|w_j^1|^2+|w_j^2|^2+\cdots |w_j^r|^2}$. 

%%%%%%%%%%%%

%%%%%%%%%%%%
\subsection{Definition and well-definedness of the obstruction classes}\label{section:3_2}

Let $X$, $Y$, $L$, $\{(U_j, z_j)\}$ and $\{(V_j, e_j, w_j)\}$ be as in the previous subsection. 
Here we drop the assumption that $L$ is semi-positive, 
and consider (inductive) linearization of the transition functions of $e_j$'s in each finite order jet. 

Assume that $\{(V_j, e_j, w_j)\}$ is a system of type $n$. 
Then the expansion of the function $t_{jk}^{-1}e_k/e_j$ around $U_{jk}$ is in the form 
\[
\frac{t_{jk}^{-1}e_k}{e_j} = 1+\sum_{|\alpha|=n} f_{kj, \alpha}(z_j)\cdot w_j^\alpha + O(|w_j|^{n+1}). 
\]
As 
\[
\frac{t_{jk}^{-1}e_k}{e_j}  \cdot \frac{t_{k \ell}^{-1}e_\ell}{e_k} \cdot \frac{t_{\ell j}^{-1}e_j}{e_\ell}  = 1,
\]
one have that $\{(U_{jk},\ \sum_{|\alpha|=n} f_{kj, \alpha}(z_j)\cdot dw_j^\alpha)\}$ satisfies the $1$-cocycle condition, where $dw_j^\alpha := \bigotimes_{\lambda=1}^r (dw_j^\lambda)^{\otimes\alpha_\lambda}$. 
We denote by $u_n(Y, X, L; \{(V_j, e_j, w_j)\})$ the class
\[
\left[\left\{\left(U_{jk},\ \sum_{|\alpha|=n} f_{kj, \alpha}(z_j)\cdot dw_j^\alpha\right)\right\}\right]
\in \check{H}^1(\{U_j\}, \mathcal{O}_Y(S^nN_{Y/X}^*)), 
\]
and call it as {\it $n$-th obstruction class} for the linearization of the transition functions of $L$. 

In the rest of this subsection, we discuss the well-definedness of this class of $H^1(Y, S^nN_{Y/X}^*)$; i.e. the dependence of this class on the choice of a system $\{(V_j, e_j, w_j)\}$ 
of type $n$. 
First we show the following: 

\begin{lemma}\label{lem:well-def_1}
Let $\{(V_j, e_j, w_j)\}$ be a system of type $n$. 
Take another local frame $\widehat{e}_j$ of $L$ on $V_j$ such that $\{(V_j, \widehat{e}_j, w_j)\}$ is also a system of type $n$. 
Assume $t_{jk}^{-1}e_k=e_j$ and 
$t_{jk}^{-1}\widehat{e}_k=\widehat{e}_j$ hold on each $U_{jk}$. 
Assume also that one of the following three conditions holds: \\
$(i)$ $n=1$, \\
$(ii)$ $N_{Y/X}$ is unitary flat and the system $\{(V_j, w_j)\}$ of local defining functions of $Y$ is of type $n$, or\\
$(iii)$ $H^0(Y, S^mN_{Y/X}^*)=0$ holds for any integer $m$ with $1\leq m\leq n-1$. \\
Then it holds that 
$u_n(Y, X, L; \{(V_j, e_j, w_j)\})=u_n(Y, X, L; \{(V_j, \widehat{e}_j, w_j)\})$. 
\end{lemma}

\begin{proof}
We let 
\[
\frac{t_{jk}^{-1}\widehat{e}_k}{\widehat{e}_j} = 1+\sum_{|\alpha|\geq n} \widehat{f}_{kj, \alpha}(z_j)\cdot w_j^\alpha 
\]
and 
\[
\frac{\widehat{e}_j}{e_j} = \sum_{|\alpha|\geq 0} A_{j, \alpha}(z_j)\cdot w_j^\alpha
\]
be the expansions. As it holds that
\[
\frac{t_{jk}^{-1}\widehat{e}_k}{\widehat{e}_j} \cdot \frac{\widehat{e}_j}{e_j}
= \frac{t_{jk}^{-1}e_k}{e_j} \cdot \frac{\widehat{e}_k}{e_k}, 
\]
we will compare the left hand side 
\begin{equation}\label{eq:welldef_1}
\frac{t_{jk}^{-1}\widehat{e}_k}{\widehat{e}_j} \cdot \frac{\widehat{e}_j}{e_j}
=\left(1+\sum_{|\alpha|\geq n} \widehat{f}_{kj, \alpha}(z_j)\cdot w_j^\alpha\right)\cdot \left(\sum_{|\alpha|\geq 0} A_{j, \alpha}(z_j)\cdot w_j^\alpha\right)
\end{equation}
and right hand side
\begin{equation}\label{eq:welldef_2}
\frac{t_{jk}^{-1}e_k}{e_j} \cdot \frac{\widehat{e}_k}{e_k}
=\left(1+\sum_{|\alpha|\geq n} f_{kj, \alpha}(z_j)\cdot w_j^\alpha\right)\cdot \left(\sum_{|\alpha|\geq 0} A_{k, \alpha}(z_k)\cdot w_k^\alpha\right)
\end{equation}
in each order. 

First, by comparing $0$-th order of the equations (\ref{eq:welldef_1}) and (\ref{eq:welldef_2}), 
one have that $A_{j, 0}=A_{k, 0}$ for each $j$ and $k$, where we are simply denoting by $0$ the multi-index $(0, 0, \dots, 0)$. 
Thus we have that there exists a constant $A_0$ such that $A_{j, 0}=A_0$ holds on each $j$. 
As both $\widehat{e}_j$ and $e_j$ are local frames, one have that $A_0\in\mathbb{C}^*:=\mathbb{C}\setminus\{0\}$. 

Next, let us observe the case where $n=1$. 
By comparing the equations (\ref{eq:welldef_1}) and (\ref{eq:welldef_2}), one have that
\[
\sum_{|\alpha|=1} A_0\cdot \widehat{f}_{kj, \alpha}(z_j)\cdot w_j^\alpha+ \sum_{|\alpha|=1} A_{j, \alpha}(z_j)\cdot w_j^\alpha
= \sum_{|\alpha|= 1} A_0f_{kj, \alpha}(z_j)\cdot w_j^\alpha
+\sum_{|\alpha|= 1} A_{k, \alpha}(z_k)\cdot w_k^\alpha + O(|w_j|^2). 
\]
Thus one have that
\[
\sum_{|\alpha|= 1} (\widehat{f}_{kj, \alpha}(z_j)-f_{kj, \alpha}(z_j))\cdot dw_j^\alpha
=\frac{1}{A_0}\cdot \left(
\sum_{|\alpha|=1} A_{j, \alpha}(z_j)\cdot dw_j^\alpha
- \sum_{|\alpha|= 1} A_{k, \alpha}(z_k)\cdot dw_k^\alpha
\right), 
\]
which proves the assertion $(i)$. 

Finally, let us assume either the assumptions $(ii)$ or $(iii)$. 
Take an integer $m$ such that $1\leq m\leq n-1$. 
As an inductive assumption, we assume that 
$A_{j, \alpha}$ is a constant for each $j$ if $|\alpha|<m$. 
Then, as it follows from a simple inductive observation that $\{(U_j, \sum_{|\alpha|=\mu}A_{j, \alpha}dw_j^\alpha)\}$ glue to define an element of $H^0(Y, S^\mu N_{Y/X}^*)$, 
one have the following again by comparing the equations (\ref{eq:welldef_1}) and (\ref{eq:welldef_2}): 
\begin{align*}
%\sum_{0\leq |\alpha|\leq m-1} A_{j, \alpha}\cdot w_j^\alpha +
\sum_{|\alpha| = m} A_{j, \alpha}(z_j)\cdot w_j^\alpha
&=%\sum_{0\leq |\alpha|\leq m-1} A_{k, \alpha}\cdot w_k^\alpha +
\sum_{|\alpha| = m} A_{k, \alpha}(z_k)\cdot w_k^\alpha + O(|w_k|^{m+1}). 
\end{align*}
Therefore we have that
$\left\{\left(U_j,\ \sum_{|\alpha| = m} A_{j, \alpha}(z_j)\cdot dw_j^\alpha\right)\right\}$ 
also glues up to define a global section of $S^mN_{Y/X}^*$. 

From the argument above, one have that 
\[
\sum_{|\alpha|= n} (\widehat{f}_{kj, \alpha}(z_j)-f_{kj, \alpha}(z_j))\cdot dw_j^\alpha
=\frac{1}{A_0}\cdot \left(
\sum_{|\alpha|=n} A_{j, \alpha}(z_j)\cdot dw_j^\alpha
- \sum_{|\alpha|= n} A_{k, \alpha}(z_k)\cdot dw_k^\alpha
\right), 
\]
from which the assertions follow. 
Note that, in order to proceed the induction in the argument above for the assertion $(ii)$, we used the constantness result for a global section of a unitary flat vector bundle, see \cite[Lemma 2.1, Remark 2.7]{K2018}. 
\end{proof}

\begin{proposition}\label{prop:well-def_main}
Let $\{(V_j, e_j, w_j)\}$ be a system of type $n$. 
Take another system $\{(\widehat{V}_j, \widehat{e}_j, \widehat{w}_j)\}$ of type $n$. 
Assume also that one of the following three conditions holds: \\
$(i)$ $n=1$, \\
$(ii)$ $N_{Y/X}$ is unitary flat and the system $\{(V_j, w_j)\}$ of local defining functions of $Y$ is of type $n$, or\\
$(iii)$ $H^0(Y, S^mN_{Y/X}^*)=0$ holds for any integer $m$ with $1\leq m\leq n-1$. \\
Then it holds that 
$u_n(Y, X, L; \{(V_j, e_j, w_j)\})=u_n(Y, X, L; \{(\widehat{V}_j, \widehat{e}_j, \widehat{w}_j)\})$. 
\end{proposition}

Based on this Proposition \ref{prop:well-def_main}, we simply denote by $u_n(Y, X, L)$ the $n$-th obstruction class when either of three assumptions in the proposition holds.

\begin{proof}
By taking refinements, we may assume that two open coverings $\{V_j\}$ and $\{\widehat{V}_j\}$ coincide. 

First, we show the proposition by assuming $w_j=\widehat{w}_j$ on each $V_j (=\widehat{V}_j)$. 
Let $\widehat{t}_{jk}$ be an element of $\mathrm{U}(1)$ such that $\widehat{t}_{jk}^{-1}\widehat{e}_k=\widehat{e}_j$ holds on each $U_{jk}$. 
As both $\{t_{jk}\}$ and $\{\widehat{t}_{jk}\}$ the transition functions of $L|_Y$, there exists 
$c_j\in\mathrm{U}(1)$ such that 
\[
\frac{\widehat{t}_{jk}}{t_{jk}} = \frac{c_k}{c_j}. 
\]
Note that here we used the fact that any unitary flat line bundle on $Y$ is trivial as unitary flat line bundle if it is analytically trivial: i.e. $H^1(Y, \mathrm{U}(1))\to H^1(Y, \mathcal{O}_Y^*)$ is injective. 
Consider a new local frame defined by $\varepsilon_j:=c_je_j$. 
Then, as it holds that 
\[
\frac{\varepsilon_k}{\varepsilon_j} = \frac{e_k}{e_j}\cdot \frac{c_k}{c_j}
=t_{jk}\cdot \frac{\widehat{t}_{jk}}{t_{jk}} = \widehat{t}_{jk}
\]
on each $U_{jk}$, it follows from Lemma \ref{lem:well-def_1} that the proof of the assertion is reduced to show that 
$u_n(Y, X, L; \{(V_j, e_j, w_j)\})=u_n(Y, X, L; \{(V_j, \varepsilon_j, w_j)\})$, which follows from 
\[
\frac{\widehat{t}_{jk}^{-1}\varepsilon_k}{\varepsilon_j} 
= \frac{c_k^{-1}}{c_j^{-1}}\cdot t_{jk}^{-1}\cdot \frac{\varepsilon_k}{\varepsilon_j}
= t_{jk}^{-1}\cdot\frac{e_k}{e_j}
\]
on each $V_{jk}$. 

Again, let $\{(V_j, e_j, w_j)\}$ and $\{(V_j, \widehat{e}_j, \widehat{w}_j)\}$ are systems of type $n$. 
From the argument above, one have that 
$u_n(Y, X, L; \{(V_j, \widehat{e}_j, \widehat{w}_j)\})=u_n(Y, X, L; \{(V_j, e_j, \widehat{w}_j)\})$. 
As it clearly holds that $u_n(Y, X, L; \{(V_j, e_j, \widehat{w}_j)\}) = u_n(Y, X, L; \{(V_j, e_j, w_j)\})$ by definition, the proposition holds. 
\end{proof}

%\begin{remark}
%これをもうすこしformalに見直してみると, 次のようになる. 
%まず$X$を$Y$の管状近傍に絞ったうえで, $L|_Y$のflat extensionを$\widetilde{L}$とする. 
%$\mathcal{L}^U:=L\otimes \widetilde{L}^{-1}$を``上田直線束"と見做そう (これが$Y$に沿って何次まで自明かをいまは考えているということ). 
%
%自然数$n$に対して, $n\leq n(\{(V_j, e_j, w_j)\})$なるシステムは, $H^0(X, \mathcal{L}^U\otimes \mathcal{O}_X/I_Y^n)$の大域切断として説明できる (ここで$Y$の定義イデアルを$I_Y$としている). 
%さらに$n+1\leq n(\{(V_j, e_j, w_j)\})$なるシステムもあるかどうかは, 
%完全列
%\[
%0 \to \mathcal{L}^U\otimes I_Y^{n+1} 
%\to \mathcal{L}^U\otimes I_Y^n
%\to \mathcal{L}^U\otimes I_Y^n/I_Y^{n+1}
%\to 0
%\]
%を考えることで誘導される完全列の内, $H^0(X, \mathcal{L}^U\otimes I_Y^{n+1})\to H^0(X, \mathcal{L}^U\otimes I_Y^n)$の逆像であるようなsystemがあるかどうかと言い換えられる. 
%こう見ると, システムの, 射
%\[
%H^0(X, \mathcal{L}^U\otimes I_Y^n)\to H^1(X, \mathcal{L}^U\otimes I_Y^n/I_Y^{n+1})
%=H^1(Y, S^nN_{Y/X}^*)
%\]
%による像を考えるのは自然で, これこそが障害類の定義に他ならない. 
%\qed
%\end{remark}

\begin{remark}
Let $Y$ be a hypersurface of $X$ such that the normal bundle $N_{Y/X}$ is unitary flat and that $L=[Y]$. 
Assume that there exists a system $\{(V_j, w_j)\}$ of local defining functions of $Y$ of type $n$: 
i.e. it holds that 
\[
t_{jk}w_k = w_j + \sum_{\nu=n+1}^\infty g_{kj, \nu}(z_j)\cdot w_j^\nu
\]
on each $V_{jk}$ for some holomorphic functions $g_{kj, \nu}$'s. 
In this case, $n$-th Ueda class $u_n(Y, X)$ is the class defined by the $1$-cocycle 
$\{U_{jk}, g_{kj, n+1}(z_j)\cdot dw_j^n\}$. 
Let $\{(V_j, e_j)\}$ be a system of local frames of $[Y]$ such that $e_j$ corresponds to the meromorphic function $1/w_j$ via the isomorphism $\mathcal{O}_X([Y])\cong \mathcal{O}_X(Y)$. 
Then, as 
\[
\frac{t_{jk}^{-1}e_k}{e_j}
= \frac{w_j}{t_{jk}w_k}
= \left(1 + \sum_{\nu=n}^\infty g_{kj, \nu+1}(z_j)\cdot w_j^\nu\right)^{-1}
= 1 - g_{kj, n+1}(z_j)\cdot w_j^n +O(w_j^{n+1}), 
\]
one have that our $n$-th obstruction class $u_n(Y, X, L)$ coincides with $n$-th Ueda class $u_n(Y, X)$ up to multiplication by a $\mathbb{C}^*$-constant. 
\end{remark}

%%%%%%%%%%%%

%%%%%%%%%%%%
\subsection{Linearization in the first order jet}\label{section:3_3}

In what follows, we assume that $L$ is semi-positive. 
Let $h$ be a $C^\infty$ Hermitian metric with semi-positive curvature. 
Denote by $\varphi_j$ the local weight function of $h$ on each  $V_j$. 
By Taylor's theorem, there exists a $C^\infty$'ly smooth functions $R^{(2)}_j(z_j, w_j)$ on each $V_j$ such that 
\[
\begin{cases}
\varphi_j(z_j, w_j) =\varphi_j^{(0)}(z_j) + \sum_{\lambda=1}^r\left(\varphi_j^{(\lambda)}(z_j)\cdot w_j^\lambda + \overline{\varphi_j^{(\lambda)}(z_j)}\cdot \overline{w_j^\lambda}\right) + R^{(2)}_j(z_j, w_j) \\
R^{(2)}_j(z_j, w_j)=O(|w_j|^2)\ \text{as}\ |w_j|\to 0. 
\end{cases}
\]
Then we have that 

\begin{align*}
\varphi_k-\varphi_j %&= 
%\varphi_k^{(0)} + \sum_{\lambda=1}^r\varphi_k^{(\lambda)}\cdot w_k^\lambda + \overline{\varphi_k^{(\lambda)}}\cdot \overline{w_k^\lambda} + R^{(2)}_k - \varphi_j^{(0)} - \sum_{\lambda=1}^r\varphi_j^{(\lambda)}\cdot w_j^\lambda - \overline{\varphi_j^{(\lambda)}}\cdot \overline{w_j^\lambda} - R^{(2)}_j \\
&= (\varphi_k^{(0)}-\varphi_j^{(0)}) + 
\sum_{\lambda=1}^r \left(\varphi_k^{(\lambda)}\cdot w_k^\lambda-\varphi_j^{(\lambda)}\cdot w_j^\lambda + \overline{\varphi_k^{(\lambda)}}\cdot \overline{w_k^\lambda}-\overline{\varphi_j^{(\lambda)}}\cdot \overline{w_j^\lambda}\right) + O(|w_j|^2). 
\end{align*}
By the equation (\ref{eq:diff_phi}), we have that
\begin{align*}
&-\sum_{|\alpha|=1} f_{kj, \alpha}(z_j)\cdot w_j^\alpha
-\sum_{|\alpha|=1} \overline{f_{kj, \alpha}(z_j)}\cdot \overline{w_j^\alpha}\\
&=(\varphi_k^{(0)}-\varphi_j^{(0)}) + 
\sum_{\lambda=1}^r (\varphi_k^{(\lambda)}\cdot w_k^\lambda-\varphi_j^{(\lambda)}\cdot w_j^\lambda) +
\sum_{\lambda=1}^r (\overline{\varphi_k^{(\lambda)}}\cdot \overline{w_k^\lambda}-\overline{\varphi_j^{(\lambda)}}\cdot \overline{w_j^\lambda}) + O(|w_j|^2), 
\end{align*}
%We let $S_{jk} = S_{jk}(z_j)$ be a matrix $((S_{jk})^\nu_\mu)_{\nu, \mu}$ such that 
%\[
%w_j^\nu = \sum_{\mu=1}^r (S_{jk})^\nu_\mu \cdot w_k^\mu. 
%\]
%As 
%\[
%\sum_{\lambda=1}^r(\varphi_k^{(\lambda)}\cdot w_k^\lambda-\varphi_j^{(\lambda)}\cdot w_j^\lambda)
%=\sum_{\lambda=1}^r\sum_{\mu=1}^r\varphi_k^{(\lambda)}\cdot (S_{kj})^\lambda_\mu w_j^\mu
%-\sum_{\lambda=1}^r\varphi_j^{(\lambda)}\cdot w_j^\lambda, 
%\]
%one have that
from which it follows that 
\begin{equation}\label{eq:1}
\varphi_k^{(0)} = \varphi_j^{(0)}
\end{equation}
and 
\begin{equation}\label{eq:2}
-\sum_{|\alpha|=1}f_{kj, \alpha}dw_j^\alpha
 = \sum_{\lambda=1}^r \varphi_k^{(\lambda)}\cdot dw_k^\lambda
- \sum_{\lambda=1}^r\varphi_j^{(\lambda)}\cdot dw_j^\lambda. 
%\sum_{\nu=1}^r\varphi_k^{(\nu)}\cdot (S_{kj})^\nu_\lambda -\varphi_j^{(\lambda)}, 
\end{equation}
%where $e_\lambda$ is the multi-index such that $|e_\lambda|=1$ and only the $\lambda$-th element is $1$. 

The equation (\ref{eq:1}) implies that $\{(U_j, \varphi^{(0)}_j)\}$ defines a global function $\varphi^{(0)}\colon Y\to \mathbb{R}$. 
As the Chern curvature of $h|_Y$ is semi-positive, it clearly holds that $\varphi^{(0)}$ is plurisubharmonic. 
The compactness of $Y$ implies that $\varphi^{(0)}$ is constant. 
By changing $h$ by multiplying a constant, we always assume that $\varphi^{(0)}\equiv 0$ in what follows. 

\begin{lemma}\label{lem:1}
For any $\lambda=1, 2, \dots, r$ and any $j$, $\varphi_j^{(\lambda)}$ is a holomorphic function on $U_j$.  
\end{lemma}

\begin{proof}
For any $\lambda=1, 2, \dots, r$, $\nu=1, 2, \dots, d$, and any $j$, we show that 
\[
(\varphi_j^{(\lambda)})_{\overline{z^\nu_j}}(z_j) := \frac{\partial}{\partial \overline{z_j^\nu}}\varphi_j^{(\lambda)}(z_j)\equiv 0
\]
holds for all $z_j\in U_j$. 
Take a point $z_j\in U_j$, and consider the Hermitian form $\langle -, -\rangle_{\lambda, \nu}\colon \mathbb{C}^2\times \mathbb{C}^2\to \mathbb{C}$ defined by 
\[
\left\langle 
\left(
    \begin{array}{c}
      a \\
      b
    \end{array}
  \right)
,\ 
\left(
    \begin{array}{c}
      c\\
	d
    \end{array}
  \right)\right\rangle_{\lambda, \nu}
:=
(a, b)
\left(
    \begin{array}{cc}
      (\varphi_j)_{w_j^\lambda \overline{w_j^\lambda}} (z_j, 0) & (\varphi_j)_{w_j^\lambda \overline{z_j^\nu}} (z_j, 0)\\
      (\varphi_j)_{z_j^\nu \overline{w_j^\lambda}} (z_j, 0) & (\varphi_j)_{z_j^\nu \overline{z_j^\nu}} (z_j, 0)
    \end{array}
  \right)
\left(
    \begin{array}{c}
      \overline{c} \\
	\overline{d}
    \end{array}
  \right). 
\]
As is clearly follows from the semi-positivity assumption for the Chern curvature of $h$, 
the Hermitian form $\langle -, -\rangle_{\lambda, \nu}$ is positive semi-definite. 
Thus one have that
\[
{\rm det}\left(
    \begin{array}{cc}
      (\varphi_j)_{w_j^\lambda \overline{w_j^\lambda}} (z_j, 0) & (\varphi_j)_{w_j^\lambda \overline{z_j^\nu}} (z_j, 0)\\
      (\varphi_j)_{z_j^\nu \overline{w_j^\lambda}} (z_j, 0) & (\varphi_j)_{z_j^\nu \overline{z_j^\nu}} (z_j, 0)
    \end{array}
  \right)\geq 0. 
\]
As $\varphi_j^{(0)}\equiv 0$, it follows that $(\varphi_j)_{z_j^\nu \overline{z_j^\nu}} (z_j, 0)=0$. 
Therefore, we have that 
$(\varphi_j)_{w_j^\lambda \overline{z_j^\nu}}(-, 0) \equiv 0$. 
The assertion follows from the calculation 
\[
(\varphi_j)_{w_j^\lambda \overline{z_j^\nu}} = (\varphi_j^{(\lambda)})_{\overline{z_j^\nu}}+ (R^{(2)}_j)_{w_j^\lambda \overline{z_j^\nu}} 
=  (\varphi_j^{(\lambda)})_{\overline{z_j^\nu}} + O(|w_j|).
\]
\end{proof}

\begin{proof}[Proof of Theorem {\ref{thm:1}}]
Assume that $L$ is semi-positive. 
We use the notation as above and show that $u_1(Y, X, L)=0$. 
Let $\widehat{e}_j$ be the section of $L$ on $V_j$ defined by 
\[
\widehat{e}_j=e_j\cdot \left(1+\sum_{\lambda=1}^r\varphi_j^{(\lambda)}(z_j)\cdot w_j^\lambda\right)
\]
for each $j$ after shrinking $V_j$ if necessary (so that $\widehat{e}_j\not=0$ at any point of $V_j$). 
By Lemma \ref{lem:1}, $\widehat{e}_j$'s are also (holomorphic) local frames. 
Theorem follows by calculating $u_1(Y, X, L)$ by using them. 
\end{proof}

Let $\widehat{e}_j$ be the one in the proof of Theorem \ref{thm:1}. 
Denote by $\widehat{\varphi}_j$ the corresponding local weight function of $h$. 
Then one have that
\begin{align*}
\widehat{\varphi}_j(z_j, w_j) 
&=-\log |\widehat{e}_j|_{h(z_j, w_j)}^2 \\
&=-\log |e_j|_{h(z_j, w_j)}^2 - \log \left|1+\sum_{\lambda=1}^r\varphi_j^{(\lambda)}(z_j)\cdot w_j^\lambda\right|^2 \\
&= \varphi_j(z_j, w_j) - \log \left|1+\sum_{\lambda=1}^r\varphi_j^{(\lambda)}(z_j)\cdot w_j^\lambda\right|^2\\
&= \varphi_j(z_j, w_j) - \sum_{\lambda=1}^r\varphi_j^{(\lambda)}(z_j)\cdot w_j^\lambda
+ \sum_{\lambda=1}^r\overline{\varphi_j^{(\lambda)}(z_j)}\cdot \overline{w_j^\lambda} + O(|w_j|^2) = O(|w_j|^2). 
\end{align*}
Thus we have that, by leaving $w_j$'s as it was and changing only $e_j$'s, 
one may assume that $\varphi_j(z_j, w_j) = O(|w_j|^2)$ holds as $|w_j|\to 0$ for each $j$. 

%%%%%%%%%%%%

%%%%%%%%%%%%
\subsection{Linearization in the second order jet when $Y$ is a hypersurface}\label{section:3_4}

In what follows, we assume that $Y$ is a hypersurface of $X$: i.e. $r=1$. 
We simply denote by $w_j$ the function $w_j^1$. 

By the argument in the previous subsection, we may assume that the expansion of $\varphi_j$ is in the form 
\[
\varphi_j(z_j, w_j) 
= \varphi_j^{(2, 0)}(z_j)\cdot w_j^2
+ \varphi_j^{(1, 1)}(z_j)\cdot |w_j|^2
+ \varphi_j^{(0, 2)}(z_j)\cdot \overline{w_j}^2
+ R^{(3)}_j(z_j, w_j), 
\]
where $R^{(3)}_j(z_j, w_j)$ is a $C^\infty$'ly smooth function on $V_j$ such that 
$R^{(3)}_j(z_j, w_j)=O(|w_j|^3)$ as $|w_j|\to 0$. 
Note that,  as $\varphi_j$ is real valued, one have that $\varphi_j^{(0, 2)}=\overline{\varphi_j^{(2, 0)}}$. 

Let $H_j$ be the $(d+1)\times (d+1)$ matrix with entries $(H_j)^a_b$ with $0\leq a, b\leq d$ defined by 
\[
(H_j)^a_b = \begin{cases}
 \frac{\partial^2 \varphi_j}{\partial w_j \partial \overline{w_j}} & \text{if}\ a=b=0 \\
 \frac{\partial^2 \varphi_j}{\partial w_j \partial \overline{z_j^b}} & \text{if}\ a=0, b>0 \\
 \frac{\partial^2 \varphi_j}{\partial z_j^a \partial \overline{w_j}} & \text{if}\ a>0, b=0 \\
 \frac{\partial^2 \varphi_j}{\partial z_j^a \partial \overline{z_j^b}} & \text{if}\ a, b>0
\end{cases}
\]
i.e. $H_j$ is the complex Hessian of $\varphi_j$. 
By the assumption that the curvature of $h$ is semi-positive, we have that $H_j$ is positive semi-definite definite. 
By a simple computation, one have that 
\[
(H_j)^a_b = \begin{cases}
\varphi_j^{(1, 1)} + O(|w_j|) & \text{if}\ a=b=0 \\
2(\varphi_j^{(2, 0)})_{\overline{z_j^b}}\cdot w_j + (\varphi_j^{(1, 1)})_{\overline{z_j^b}}\cdot \overline{w_j} + O(|w_j|^2)  & \text{if}\ a=0, b>0 \\
(\varphi_j^{(1, 1)})_{z_j^a}\cdot w_j + 2(\varphi_j^{(0, 2)})_{z_j^a}\cdot \overline{w_j} + O(|w_j|^2)  & \text{if}\ a>0, b=0 \\
(\varphi_j^{(2, 0)})_{z_j^a\overline{z_j^b}}\cdot w_j^2
+ (\varphi_j^{(1, 1)})_{z_j^a\overline{z_j^b}}\cdot |w_j|^2
+ (\varphi_j^{(0, 2)})_{z_j^a\overline{z_j^b}}\cdot \overline{w_j}^2 + O(|w_j|^3) & \text{if}\ a, b>0
\end{cases}.
\]
As $H_j$ is positive semi-definite, one have that $(H_j)^0_0$ is non-negative. 
Thus one have that $\varphi_j^{(1, 1)}$ is also non-negative. 
%By further applying by Lemma \ref{lem_positive_function}, one have the following: 

\begin{lemma}\label{prop:11_psh_pre}
The function $\varphi_j^{(1, 1)}$ is plurisubharmonic on $U_j$. 
\end{lemma}

\begin{proof}
Take $d$ complex constants $\xi_1, \dots, \xi_d\in \mathbb{C}$. 
Denote by $v_\xi$ the vector $(0, \xi_1, \xi_2, \dots, \xi_d)\in \mathbb{C}^{d+1}$. 
Then, as $H_j$ is positive semi-definite, one have that 
\begin{align*}
v_\xi H_j {}^t\!\overline{v_\xi}
&= \sum_{a=1}^d\sum_{b=1}^d \xi_a(H_j)^a_b\overline{\xi_b},
\end{align*}
is non-negative. 
By applying Lemma \ref{lem_positive_function} to this function, one have that 
\[
\sum_{\nu=1}^d\sum_{\mu=1}^d \xi_\nu (\varphi_j^{(1, 1)})_{z_j^\nu\overline{z_j^\mu}}\overline{\xi_\mu}\geq 0, 
\]
from which the assertion follows. 
\end{proof}

\begin{proposition}\label{prop:11_psh}
Assume that $\varphi_j^{(1, 1)}\not\equiv 0$. 
Then the function $z_j\mapsto \log \varphi_j^{(1, 1)}(z_j)$ is plurisubharmonic. 
\end{proposition}

\begin{proof}
For a positive real number $\varepsilon$, we consider the function 
$\psi_\varepsilon(z_j):=\log \left(\varphi_j^{(1, 1)}(z_j)+\varepsilon\right)$. 
As $\psi_\varepsilon$'s monotonically approximate the function $\log \varphi_j^{(1, 1)}(z_j)$ from above, it is sufficient to show that each $\psi_\varepsilon$ is plurisubharmonic. 

Let $\xi = (\xi_1, \xi_2, \dots, \xi_d)\in \mathbb{C}^d$ be a vector. 
As 
\begin{align*}
\partial \overline{\partial} \psi_\varepsilon
= &\frac{1}{(\varphi_j^{(1, 1)}+\varepsilon)^2}\sum_{\nu=1}^d\sum_{\mu=1}^d 
\left(\varphi_j^{(1, 1)}\cdot(\varphi_j^{(1, 1)})_{z_j^\nu\overline{z_j^\mu}}
-(\varphi_j^{(1, 1)})_{z_j^\nu}\cdot (\varphi_j^{(1, 1)})_{\overline{z_j^\mu}}\right)
dz_j^\nu\wedge d\overline{z_j^\mu}\\
&+\frac{\varepsilon}{(\varphi_j^{(1, 1)}+\varepsilon)^2}\sum_{\nu=1}^d\sum_{\mu=1}^d 
(\varphi_j^{(1, 1)})_{z_j^\nu\overline{z_j^\mu}}\,
dz_j^\nu\wedge d\overline{z_j^\mu},\\
\end{align*}
it follows from Lemma \ref{prop:11_psh_pre} that 
it is sufficient to show the inequality
\begin{equation}\label{ineq:wts1}
\sum_{\nu=1}^d\sum_{\mu=1}^d \xi_\nu\overline{\xi_\mu} \cdot
\left(
\varphi_j^{(1, 1)}\cdot(\varphi_j^{(1, 1)})_{z_j^\nu\overline{z_j^\mu}}
-(\varphi_j^{(1, 1)})_{z_j^\nu}\cdot (\varphi_j^{(1, 1)})_{\overline{z_j^\mu}}
\right)\geq 0. 
\end{equation}

We set $u:=(1, 0, 0, \dots, 0)\in \mathbb{C}^{d+1}$ and 
$v_\xi:=(0, \xi_1, \xi_2, \dots, \xi_d)\in \mathbb{C}^{d+1}$. 
As the quadratic form 
\[
\left\langle 
\left(
    \begin{array}{c}
      a \\
      b
    \end{array}
  \right)
,\ 
\left(
    \begin{array}{c}
      c\\
	d
    \end{array}
  \right)\right\rangle:=(au+bv_\xi)H_j{}^t\!\overline{(cu+dv_\xi)}
\]
is semi-positive definite, one have that 
\[
{\rm det}
\left(
    \begin{array}{cc}
      uH_j {}^t\!\overline{u} & u H_j {}^t\!\overline{v_\xi}\\
      v_\xi H_j {}^t\!\overline{u} & v_\xi H_j {}^t\!\overline{v_\xi}
    \end{array}
  \right)\geq 0. 
\]
As it holds that 
$uH_j {}^t\!\overline{u} = (H_j)^0_0 = \varphi_j^{(1, 1)} + O(|w_j|)$ 
and
\begin{align*}
uH_j {}^t\!\overline{v_\xi}
&= \sum_{\nu=1}^d (H_j)^0_\nu\overline{\xi_\nu} \\
&= \sum_{\nu=1}^d \overline{\xi_\nu}\cdot \left( 2(\varphi_j^{(2, 0)})_{\overline{z_j^\nu}}\cdot w_j + (\varphi_j^{(1, 1)})_{\overline{z_j^\nu}}\cdot \overline{w_j} \right) + O(|w_j|^2) \\
&= 2\left(\sum_{\nu=1}^d \overline{\xi_\nu}\cdot(\varphi_j^{(2, 0)})_{\overline{z_j^\nu}}\right)\cdot w_j + \left(\sum_{\nu=1}^d \overline{\xi_\nu}\cdot(\varphi_j^{(1, 1)})_{\overline{z_j^\nu}}\right)\cdot \overline{w_j} + O(|w_j|^2), 
\end{align*}
%and 
%\begin{align*}
%v_\xi H_j {}^t\!\overline{v_\xi}
%&= \sum_{a=0}^d\sum_{b=0}^d \xi_a(H_j)^a_b\overline{\xi_b},
%\end{align*}
it follows from Lemma \ref{lem_positive_function} that 
\begin{align*}
&\varphi_j^{(1, 1)}\cdot \sum_{\nu=1}^d\sum_{\mu=1}^d \xi_\nu\overline{\xi_\mu} \cdot (\varphi_j^{(1, 1)})_{z_j^\nu\overline{z_j^\mu}}
-\left|\sum_{\nu=1}^d \overline{\xi_\nu}\cdot(\varphi_j^{(1, 1)})_{\overline{z_j^\nu}}\right|^2
- 4\cdot 
\left|\sum_{\nu=1}^d \overline{\xi_\nu}\cdot(\varphi_j^{(2, 0)})_{\overline{z_j^\nu}}\right|^2 \\
&= \sum_{\nu=1}^d\sum_{\mu=1}^d \xi_\nu\overline{\xi_\mu} \cdot
\left(
\varphi_j^{(1, 1)}\cdot(\varphi_j^{(1, 1)})_{z_j^\nu\overline{z_j^\mu}}
-(\varphi_j^{(1, 1)})_{z_j^\nu}\cdot (\varphi_j^{(1, 1)})_{\overline{z_j^\mu}}
\right)
- 4\cdot 
\left|\sum_{\nu=1}^d \overline{\xi_\nu}\cdot(\varphi_j^{(2, 0)})_{\overline{z_j^\nu}}\right|^2
\end{align*}
is non-negative, which proves the inequality (\ref{ineq:wts1}). 
\end{proof}

From Proposition \ref{prop:11_psh} and the last inequality of the  proof of it, 
one can easily deduce the following: 

\begin{theorem}\label{thm:2}
Let $X$ be a complex manifold and $Y$ be a non-singular compact hypersurface of $X$ which is K\"ahler. 
Let $L$ be a semi-positive line bundle on $X$ such that $L|_Y$ is topologically trivial. 
Assume that the conormal bundle $N_{Y/X}^{-1}$ is not pseudo-effective. 
Then there exists a system of local trivializations $\{(\widehat{V}_j, \widehat{e}_j, w_j)\}$ of type $3$. %$L$ such that 
%$n=n(\{(\widehat{V}_j, \widehat{e}_j)\})$ is larger than $2$. 
Especially, it holds that $u_1(Y, X, L)=u_2(Y, X, L)=0$. 
\end{theorem}

Note that the second obstruction class $u_2(Y, X, L)$ is well-defined in this case, since $H^0(Y, N_{Y/X}^*)=0$ (Here we used Proposition \ref{prop:well-def_main} $(iii)$). 

\begin{proof}[Proof of Theorem {\ref{thm:2}}]
Assume that $\varphi_j^{(1, 1)}\not\equiv 0$. 
It follows from the equation (\ref{eq:diff_phi}) with $n=2$ that 
$\{(U_j, \varphi_j^{(1, 1)}dw_j\wedge d\overline{w}_j)\}$ 
glue up to define a global form on $Y$. 
By Proposition \ref{prop:11_psh}, one have that this form is non-zero outside of a pluripolar subset of $Y$. 
Thus, again by using Proposition \ref{prop:11_psh},  it turns out that one may regard the system $\{(U_j, \varphi_j^{(1, 1)})\}$ as a singular Hermitian metric on $N_{Y/X}$, say $h_N$, and that the curvature current $\sqrt{-1}\Theta_{h_N^{-1}}$ of the dual metric $h_N^{-1}$ on $N_{Y/X}^{-1}$ is semi-positive. 
%As is well-known, the existence of a semi-positively curved singular Hermitian metric implies that the pseudo-effectivity. ←compact Ka でpeの定義どうするってなるから余計なこと言わないことにしよう
Thus, as $N_{Y/X}^{-1}$ is not pseudo-effective, one have that $\varphi_j^{(1, 1)}\equiv 0$. 

If follows from the same argument as in the proof of Proposition \ref{prop:11_psh} that 
$(\varphi_j^{(2, 0)})_{\overline{z_j^\nu}}\equiv 0$ 
for any $\nu=1, 2, \dots, d$. 
Therefore we have that $\varphi^{(2, 0)}$ is a holomorphic function on $U_j$. 

Let $\widehat{e}_j$ be the section of $L$ on $V_j$ defined by 
\[
\widehat{e}_j=e_j\cdot \left(1+\varphi_j^{(2, 0)}(z_j)\cdot w_j^2\right)
\]
for each $j$ after shrinking $V_j$ if necessary (so that $\widehat{e}_j\not=0$ at any point of $V_j$). 
Then, by a simple computation, one have that the system $\{(V_j, \widehat{e}_j, w_j)\}$ is of type $3$. 
\end{proof}

%As $\varphi_{(2, 0)}$ is holomorphic, $\widehat{e}_j$ is also a (holomorphic) local frame. 
Let $\widehat{e}_j$ be as in the proof of Theorem \ref{thm:2}. 
Denote by $\widehat{\varphi}_j$ the corresponding local weight function of $h$. 
Then one have that
\begin{align*}
\widehat{\varphi}_j(z_j, w_j) 
&=-\log |\widehat{e}_j|_{h(z_j, w_j)}^2 \\
&=-\log |e_j|_{h(z_j, w_j)}^2 - \log \left|1+\varphi_j^{(2, 0)}(z_j)\cdot w_j^2\right|^2 \\
&= \varphi_j(z_j, w_j) - \log \left|1+\varphi_j^{(2, 0)}(z_j)\cdot w_j^2\right|^2\\
&= O(|w_j|^3) 
\end{align*}
as $|w_j|\to 0$. 
Thus we have that, by leaving $w_j$'s as it was and changing only $e_j$'s, 
one may assume that $\varphi_j(z_j, w_j) = O(|w_j|^3)$ holds as $|w_j|\to 0$ for each $j$ when $r=1$ and $N_{Y/X}^{-1}$ is not pseudo-effective. 
The similar argument also runs when $r=1$, $N_{Y/X}$ is topologically trivial, and $L=[Y]$, which will be generalized in the following subsection. 

%%%%%%%%%%%%

%%%%%%%%%%%%
\subsection{Linearization in the higher order jets when $Y$ is a hypersurface with topologically trivial normal bundle}\label{section:3_5}
Fix a positive integer $n$. 
In this section, we assume that $Y$ is a hypersurface of $X$ (i.e. $r=1$) and that $N_{Y/X}$ is a unitary flat line bundle. 
In what follows, we let $L$ be the line bundle $[Y]$. 

On each $V_j$, we use the local frame $e_j$ of $L$ which corresponds to the meromorphic function $1/w_j$ via the isomorphism $\mathcal{O}_X([Y])\cong \mathcal{O}_X(Y)$, where $w_j$ is a local holomorphic defining function of $Y$ in $V_j$. 
In this case, as is explained in \S \ref{section:2_1}, $\{(V_j, w_j\cdot e_j)\}$ patches to each other to define the canonical section $f_Y\in H^0(V, [Y])$. 
In what follows, whenever we change a system of the local frames $e_j$'s, we will also change a system of local defining functions $w_j$'s 
so that the canonical section $f_Y$ itself never changes. 
Let $h$ be a $C^\infty$ Hermitian metric on $[Y]$ with semi-positive curvature and $\varphi_j$ be the local weight function with respect to the local frame $e_j$. 

Consider the set $Z(Y, X, h; f_Y)$ of all positive integers $n$ which satisfies the following condition: 
There exists a system $\{(V_j, e_j, w_j)\}$ of type $n$ with $f_Y=w_j\cdot e_j$ such that the corresponding local weight function $\varphi_j$ satisfies $\varphi_j(z_j, w_j)=O(|w_j|^n)$ as $|w_j|\to 0$ for each $j$. 
It clearly follows from the arguments in \S \ref{section:3_3} and Theorem \ref{thm:1} that 
$1, 2 \in Z(Y, X, h; f_Y)$. 

Take an element $n\in Z(Y, X, h; f_Y)$ and a system $\{(V_j, e_j, w_j)\}$ of type $n$ as above. 
Let 
\[
\varphi_j(z_j, w_j) = 
%\sum_{p, q> 0, p+q<n}A_j^{(p, q)}\cdot w_j^p\overline{w_j}^q + 
\sum_{p, q\geq 0, p+q=n}\varphi_j^{(p, q)}(z_j)\cdot w_j^p\overline{w_j}^q 
+R_j^{(n+1)}(z_j, w_j) 
\]
be the expression obtained by considering Taylor expansion of $\varphi_j$, 
where $R_j^{(n+1)}(z_j, w_j)$ is a smooth function with $R_j^{(n+1)}(z_j, w_j) = O(|w_j|^{n+1})$ as $|w_j|\to 0$. 
Note that, by equation (\ref{eq:diff_phi}), one have that 
$\{(U_j, \varphi_j^{(p, q)}(z_j)\,dw_j^p\otimes d\overline{w_j}^q)\}$ patch to define a global section of $N_{Y/X}^{-p+q}$ for each pair $(p, q)$ with $p, q>0$ and $p+q=n$, 
and that $t_{jk}^{-n}\cdot \varphi_k^{(n, 0)}(z_k) - \varphi_j^{(n, 0)}(z_j) = -f_{kj, n}(z_j)$ on each $U_{jk}$, where we are letting
\[
\frac{t_{jk}^{-1}e_k}{e_j} = 1+ f_{kj, n}(z_j)\cdot w_j^n + O(w_j^{n+1})
\]
be the expansion. 

\begin{lemma}\label{lem:higher_inductive}
For  $(p, q) = (n-1, 1), (n-2, 2), \dots, (2, n-2), (1, n-1)$, 
there exists a constant $A_j^{(p, q)}\in\mathbb{C}$ such that $\varphi_j^{(p, q)}\equiv A_j^{(p, q)}$ holds on each $U_j$. 
Moreover, it hold that $A_j^{(p, q)}$ is non-negative if $p=q$, and that $A_j^{(p, q)}=0$ if the line bundle $N_{Y/X}^{-p+q}$ is not analytically trivial. 
\end{lemma}

\begin{proof}
Let $M_j:=((M_j)^a_b)_{a, b}$ be the $d \times d$-matrix with entries
\[
(M_j)^\nu_\mu
= \frac{\partial^2 \varphi_j}{\partial z_j^\nu \partial \overline{z_j^\mu}} (z_j) 
= \sum_{p, q\geq 0, p+q=n}(\varphi_j^{(p, q)})_{z_j^\nu\overline{z_j^\mu}}(z_j)\cdot w_j^p\overline{w_j}^q 
+O(|w_j|^{n+1})
\]
at a point $z_j\in U_j$. 
As the curvature of $h$ is semi-positive, one have that $M_j$ is also positive semi-definite. 
Take an element $\xi=(\xi_1, \xi_2, \dots, \xi_d)\in \mathbb{C}^d$.
Then it easily follows from the semi-positivity that the sum
\[
\sum_{p, q\geq 0, p+q=n} 
\left(\sum_{\nu=1}^d \sum_{\mu=1}^d\xi_\nu\overline{\xi_\mu}\cdot (\varphi_j^{(p, q)})_{z_j^\nu\overline{z_j^\mu}}\right)\cdot w_j^p\overline{w_j}^q 
\] 
is non-negative. 

Assume that $n$ is even. Set $n=2m$. 
By Lemma \ref{lem_positive_function}, one have that 
\[
\sum_{\nu=1}^d \sum_{\mu=1}^d\xi_\nu\overline{\xi_\mu}\cdot (\varphi_j^{(m, m)})_{z_j^\nu\overline{z_j^\mu}} \geq 0. 
\]
As $\xi$ is arbitrary chosen, one have that the function $\varphi_j^{(m, m)}$ is plurisubharmonic. 
As $\varphi_j^{(m, m)}$ is a global function on a compact complex manifold $Y$, one have that it is constant. 
Therefore one obtain that $(\varphi_j^{(m, m)})_{z_j^\nu\overline{z_j^\mu}} \equiv 0$. 
Thus, again by Lemma \ref{lem_positive_function}, one have that 
\[
\sum_{\nu=1}^d \sum_{\mu=1}^d\xi_\nu\overline{\xi_\mu}\cdot (\varphi_j^{(p, q)})_{z_j^\nu\overline{z_j^\mu}} \equiv 0
\] 
holds for any element $\xi=(\xi_1, \xi_2, \dots, \xi_d)\in \mathbb{C}^d$, and any pair $(p, q)$ with $p, q\geq 0$ and $p+q=n$ in both the cases where $n$ is even/odd. 
As one can regard $\{(U_j, \varphi_j^{(p, q)})\}$'s as pluriharmonic sections of unitary flat line bundles if $(p, q)\not=(n, 0), (0, n)$, one can repeat the same argument as we have done for $(p, q)=(m, m)$ to obtain that each $\varphi_j^{(p, q)}$ is constant. 
When $n=2m$, the non-negativity of $\varphi_j^{(m, m)}$ follows by applying Lemma \ref{lem_positive_function} to the function $(\varphi_j)_{w_j\overline{w_j}}$. 
When there exists $j$ such that $\varphi_j^{(p, q)}\not\equiv 0$, we have that $\varphi_k^{(p, q)}\not\equiv 0$ for any $k$, which means that $N_{Y/X}^{-p+q}$ has a global nowhere vanishing section. 
\end{proof}

By using the constants as in Lemma \ref{lem:higher_inductive}, one can rewrite the expansion as 
\[
\varphi_j(z_j, w_j) = 
\sum_{p, q> 0, p+q = n}A_j^{(p, q)}\cdot w_j^p\overline{w_j}^q 
+ \varphi_j^{(n, 0)}(z_j)\cdot w_j^n
+ \varphi_j^{(0, n)}(z_j)\cdot \overline{w_j}^n 
+R_j^{(n+1)}(z_j, w_j). 
\]
Note that $\varphi_j^{(0, n)} = \overline{\varphi_j^{(n, 0)}}$ holds, since $\varphi_j$ is real valued. 

\begin{lemma}\label{lem:hol_higher}
The function $\varphi_j^{(n, 0)}$ is holomorphic. 
\end{lemma}

\begin{proof}
Let $H_j$ be the $(d+1)\times (d+1)$ matrix with entries $(H_j)^a_b$ with $0\leq a, b\leq d$ defined by 
\[
(H_j)^a_b = \begin{cases}
 \frac{\partial^2 \varphi_j}{\partial w_j \partial \overline{w_j}} & \text{if}\ a=b=0 \\
 \frac{\partial^2 \varphi_j}{\partial w_j \partial \overline{z_j^b}} & \text{if}\ a=0, b>0 \\
 \frac{\partial^2 \varphi_j}{\partial z_j^a \partial \overline{w_j}} & \text{if}\ a>0, b=0 \\
 \frac{\partial^2 \varphi_j}{\partial z_j^a \partial \overline{z_j^b}} & \text{if}\ a, b>0
\end{cases}
\]
i.e. $H_j$ is the complex Hessian of $\varphi_j$. 
By the assumption that the curvature of $h$ is semi-positive, we have that $H_j$ is semi-positive definite. 
By a simple computation, one have that 
\[
(H_j)^a_b = \begin{cases}
\sum_{p, q> 0, p+q = n}pqA_j^{(p, q)}\cdot w_j^{p-1}\overline{w_j}^{q-1}  + O(|w_j|^{n-1}) & \text{if}\ a=b=0 \\
n(\varphi_j^{(n, 0)})_{\overline{z_j^b}}\cdot w_j^{n-1} + O(|w_j|^n)  & \text{if}\ a=0, b>0 \\
n(\varphi_j^{(0, n)})_{z_j^a}\cdot \overline{w_j}^{n-1} + O(|w_j|^n)  & \text{if}\ a>0, b=0 \\
(\varphi_j^{(n, 0)})_{z_j^a\overline{z_j^b}}\cdot w_j^n
+ (\varphi_j^{(0, n)})_{z_j^a\overline{z_j^b}}\cdot \overline{w_j}^n 
+O(|w_j|^{n+1})  & \text{if}\ a, b>0.
\end{cases}
\]

We set $u:=(1, 0, 0, \dots, 0)\in \mathbb{C}^{d+1}$ and 
$v_\xi:=(0, \xi_1, \xi_2, \dots, \xi_d)\in \mathbb{C}^{d+1}$. 
As the quadratic form 
\[
\left\langle 
\left(
    \begin{array}{c}
      a \\
      b
    \end{array}
  \right)
,\ 
\left(
    \begin{array}{c}
      c\\
	d
    \end{array}
  \right)\right\rangle:=(au+bv_\xi)H_j{}^t\!\overline{(cu+dv_\xi)}
\]
is semi-positive definite, one have that 
\[
{\rm det}
\left(
    \begin{array}{cc}
      uH_j {}^t\!\overline{u} & u H_j {}^t\!\overline{v_\xi}\\
      v_\xi H_j {}^t\!\overline{u} & v_\xi H_j {}^t\!\overline{v_\xi}
    \end{array}
  \right)\geq 0. 
\]
As it hold that 
\[
{\rm det}
\left(
    \begin{array}{cc}
      uH_j {}^t\!\overline{u} & u H_j {}^t\!\overline{v_\xi}\\
      v_\xi H_j {}^t\!\overline{u} & v_\xi H_j {}^t\!\overline{v_\xi}
    \end{array}
  \right) = O(|w_j|^{2n-2})
\]
and the coefficient of $|w_j|^{2n-2}$ in Taylor expansion of the left hand side is 
\[
-n^2\left|\sum_{\nu=1}^d\overline{\xi_\nu}(\varphi_j^{(n, 0)})_{\overline{z_j^\nu}}\right|^2, 
\]
the assertion follows from Lemma \ref{lem_positive_function}. 
\end{proof}

By using these lemmata, one have the following: 

\begin{proposition}\label{prop:higher_main}
Let $n$ be an element of $Z(Y, X, h; f_Y)$ and $\{(V_j, e_j, w_j)\}$ be a system of type $n$ with $f_Y=w_j\cdot e_j$ such that the corresponding local weight function $\varphi_j$ satisfies $\varphi_j(z_j, w_j)=O(|w_j|^n)$ as $|w_j|\to 0$ for each $j$. Then the following holds: \\
$(i)$ $u_n(Y, X)=u_n(Y, X, L)=0$. \\
$(ii)$ When $n$ is odd, it holds that $n+1\in Z(Y, X, h; f_Y)$. \\
$(iii)$ When $n$ is even and $A_j^{(m, m)}\equiv 0$ for $m:=n/2$, it holds that $n+1\in Z(Y, X, h; f_Y)$. 
\end{proposition}

\begin{proof}
Let $\widehat{e}_j$ be the section of $L$ on $V_j$ defined by 
\[
\widehat{e}_j=e_j\cdot \left(1+\varphi_j^{(n, 0)}(z_j)\cdot w_j^n\right)
\]
for each $j$ after shrinking $V_j$ if necessary (so that $\widehat{e}_j\not=0$ at any point of $V_j$). 
Denote by $\widehat{w}_j$ the corresponding defining function of $U_j$ on $V_j$: i.e. 
\[
\widehat{w}_j = w_j\cdot \left(1+\varphi_j^{(n, 0)}(z_j)\cdot w_j^n\right)^{-1}. 
\]
Then, by a simple calculation and Lemma \ref{lem:hol_higher}, one have that the system $\{(V_j, w_j)\}$ is of type $n+1$, which shows the assertion $(i)$. 
In what follows we assume that $\{(V_j, w_j)\}$ is of type $n+1$ by replacing $w_j$ with $\widehat{w}_j$. 

When $n$ is odd or when $n=2m$ is even and $A_j^{(m, m)}\equiv 0$, 
it follows by applying Lemma \ref{lem_positive_function} to the function $(\varphi_j)_{w_j\overline{w_j}}$ that $A^{(p, q)}\equiv 0$ for any $(p, q)$ with $p+q=n$ and $p, q> 0$, from which the assertions $(ii)$ and $(iii)$ hold. 
\end{proof}

\begin{proof}[Proof of Theorem {\ref{thm:flat_u2}}]
Theorem follows from Theorem \ref{thm:1}, 
Proposition \ref{prop:higher_main} $(i)$, 
and the fact that $2\in Z(Y, X, h; f_Y)$. 
\end{proof}

Note that the set $Z(Y, X, h; f_Y)$ need not to coincide with $\mathbb{Z}_{>0}$ even if the pair $(Y, X)$ is of infinite type. 
For example, let $X$ be a surface and $Y\subset X$ be a holomorphically embedded non-singular compact curve with topologically trivial normal bundle such that the pair $(Y, X)$ is of type $(\beta')$ or $(\beta'')$ in the classification of \cite[\S 5]{U}; 
i.e. there exists an open covering $\{U_j\}$ of $Y$ and a local defining function $w_j$ of $Y$ on a neighborhood $V_j$ of $U_j$ such that $t_{jk}w_k=w_j$ holds on each $V_{jk}$ ($t_{jk}\in\mathrm{U}(1)$). 
Consider a $C^\infty$ Hermitian metric $h$ of $[Y]$ whose local weight functions $\varphi_j$ on each $V_j$ with respect to the local frame $e_j$ which corresponds to the meromorphic function $1/w_j$ satisfies 
$\varphi_j = |w_j|^2$. 
In this case, $3\not\in Z(Y, X, h; f_Y)$ whereas the pair $(Y, X)$ is of infinite type. 

%%%%%%%%%%%%

%%%%%%%%%%%%%%%%%%%%%%%%%%%%%%%%%%%%%%%%%%%

%%%%%%%%%%%%%%%%%%%%%%%%%%%%%%%%%%%%%%%%%%%
\section{Applications}\label{section:4}

%%%%%%%%%%%%
\subsection{Proof of Theorem {\ref{thm:main_1}}}\label{section:4_1}

For proving Theorem \ref{thm:main_1}, let us first explain our variant of Grauert's example (see also \cite[Problem 2.2]{FT}). 

\begin{example}\label{ex:nefbig}
Let $C$ be a Riemann surface of genus $2$. 
%Set $F:=K_C$ (in our model, the degree of $F$ is not $1$, but $2$). 
Let $\{U_j\}$ be a finite Stein cover of $C$ and $z_j$ be the coordinate of $U_j$ which comes from the standard coordinate of a connected component of the inverse image of $U_i$ by the universal covering $\mathbb{H}\to C$ ($\mathbb{H}:=\{z\in \mathbb{C}\mid {\rm Im}\,z>0\}$). 
Define a transition function $g_{jk}\colon U_{jk}\to \mathbb{C}^*$ of $K_C$ by 
$dz_j=g_{jk}(z_k)\cdot dz_k$. 
Set 
\[
\rho_j(z_j):=\frac{\sqrt{-1}}{z_j-\overline{z_j}}=\frac{1}{2\cdot {\rm Im}\,z_j}
\]
on each $U_i$ and
\[
\xi_{jk}:=\rho_jdz_j-\rho_kdz_k = (g_{jk}\cdot\rho_j - \rho_k)\cdot dz_k
\]
on each $U_{jk}$. 
As 
\[
\overline{\partial} (\rho_j\cdot dz_j) =
\frac{\sqrt{-1}dz_j\wedge d\overline{z_j}}{4({\rm Im} z_j)^2}
\]
holds on each $U_j$ and 
the $(1, 1)$-forms as the right hand side glue up to define a non-trivial element of $H^1(C, K_C)$, 
one have that each $\xi_{jk}$ is $\overline{\partial}$-closed and that 
\[
\xi:=[\{(U_{jk}, \xi_{jk})\}]\in \check{H}^1(\{U_j\}, K_C)
\]
is a non-trivial element (it follows by considering the \v{C}ech-Dolbeault correspondence). 

We let $X$ be the ruled surface which is the natural compactification of the affine bundle $\widetilde{X}:=\bigcup_jK_{U_j}/\sim$ by considering the infinity secion, 
where ``$\sim$" is the relation generated by the following: 
$(x_j, v_jdz_j)\sim (x_k, v_kdz_k)$ holds if and only if $x_j=x_k\in U_{jk}$ and $v_jdz_j=v_kdz_k+\xi_{jk}$, or equivalently, 
$g_{jk}\cdot v_j = v_k + (g_{jk}\cdot\rho_j - \rho_k)$. 
Denote by $Y\subset X$ the infinity section. 
We have that $N_{Y/X}\cong K_{C}^{-1}$ and $(Y^2)=-2$. 
Denote by $\pi\colon X\to C$ the projection. 
We let $L$ be a line bundle defined by $L:=[Y]\otimes \pi^*K_C$. 
Note that $(L^2)=2$ and $(L. Y)=0$. 
Note also that it is easily observed that $L$ is nef and big. 

Regard the fiber coordinate function $v_j$ as a local frame of the line bundle $[Y]$, 
and $dz_j$ as of $K_C$. 
Then naturally one can regard $e_j:= v_j \otimes \pi^*dz_j$ as a local frame of $L$. 
In what follows, we use $w_j:=v_j^{-1}$ as a local defining function of $Y$ and $(z_j, w_j)$ as local coordinates on a neighborhood of $\pi^{-1}(U_j)\cap Y$. 
Note that 
\[
\frac{e_k}{e_j} = \frac{\pi^*dz_k}{\pi^*dz_j}\cdot \frac{v_k}{v_j}
= g_{jk}^{-1}\cdot \frac{g_{jk}v_j-(g_{jk}\cdot\rho_j - \rho_k)}{v_j}
= 1-(\rho_j - g_{jk}^{-1}\cdot \rho_k)\cdot w_j
\]
holds. 
%As one have that
%\begin{align*}
%w_k=w_j\cdot \frac{w_k}{w_j}= w_j\cdot \frac{v_j}{v_k}
%&= w_j\cdot\frac{v_j}{g_{jk}v_j-(g_{jk}\cdot\rho_j - \rho_k)}\\
%&= w_j\cdot\left(g_{jk}-(g_{jk}\cdot\rho_j - \rho_k)\cdot w_j\right)^{-1} \\
%&= g_{jk}^{-1}w_j\cdot \left(1-(\rho_j - g_{jk}^{-1}\cdot\rho_k)\cdot w_j\right)^{-1} \\
%&= g_{jk}^{-1}w_j + (g_{jk}\cdot\rho_j - \rho_k)\cdot w_j^2 + O(w_j^3), 
%\end{align*}
%なので, 各$U_{jk}$上では$dw_k=g_{jk}dw_j$である. よって
Therefore one have that the first obstruction class 
$u_1(Y, X, L) \in H^1(Y, N_{Y/X}^{-1})$ coincides with the class $\xi\in H^1(Y, K_C)$ via the isomorphism between $N_{Y/X}^{-1}$ and $K_C$ induced by $dw_j\mapsto dz_j$. 
Thus it follows that $u_1(Y, X, L)\not=0$. 
\qed
\end{example}

\begin{proof}[Proof of Theorem {\ref{thm:main_1}}]
We show that the triple $(Y, X, L)$ described in Example \ref{ex:nefbig} satisfies that $L$ is not semi-positive. 
Assume that $L$ is semi-positive. 
Then, by Theorem \ref{thm:1}, one have that $u_1(Y, X, L)=0$, which contradicts to the fact mentioned in Example \ref{ex:nefbig}. 
\end{proof}

\begin{remark}\label{rmk:begz}
On higher dimensional manifolds, the existence of nef, big and non semi-positive line bundle have already been shown by \cite{D} and \cite{BEGZ}. 
Let us investigate \cite{BEGZ}'s version of such an example (\cite[Example 5.4]{BEGZ}) here. 
In this example, the line bundle $L$ is on the total space of the projective plane bundle $p\colon X\to C$ over an elliptic surface $C$ which corresponds to a vector bundle $E\oplus A^{-1}$, where $E$ is the non-trivial extension of the trivial line bundle by the trivial line bundle on $C$, and $A$ is an ample line bundle. 
The line bundle $L$ is the relative $\mathcal{O}(1)$-bundle (see also \cite[Example 4.2]{K2}). 
Let $Y$ be the image of the section of $p$ which corresponds to the trivial subbundle of rank $1$ included in $E$. Then one can easily calculate that the first obstruction class $u_1(Y, X, L)$ coincides with $(\xi, 0)\in H^1(C, \mathcal{O}_C)\oplus H^1(C, \mathcal{O}_C(A^{-1}))$ via the isomorphism $p|_Y\colon Y\cong C$, where $\xi$ is the extension class of $0\to \mathcal{O}_C\to \mathcal{O}_C(E)\to \mathcal{O}_C\to 0$. 
\end{remark}

%%%%%%%%%%%%

%%%%%%%%%%%%
\subsection{Proof of Theorem {\ref{thm:main_2}}}\label{section:4_2}

Here we show Theorem \ref{thm:main_2}. 
As $d=1$ under the configuration of this theorem, we denote $z_j^1$ simply by $z_j$. 

According to \cite[Theorem 1.1]{K3}, it follows from the semi-positivity of $[Y]$ that the pair $(Y, X)$ is of infinite type. 
When $N_{Y/X}$ is a torsion element of the Picard variety, it follows from \cite[Theorem 3]{U} that there exists a neighborhood $V$ of $Y$ in $X$ and a proper surjective holomorphic map $\pi\colon V\to \Delta$ with $Y=\pi^{-1}(0)$, where $\Delta$ is the unit disc of $\mathbb{C}$ ($Y$ may be a multiple fiber). 
In this case, the assertion $(ii)$ holds by regarding this fibration as a foliation and by considering the maximal principle on each fiber (=leaf). 

Therefore, the problem is reduced to the case where $N_{Y/X}$ is non-torsion in the Picard variety. 
As is mentioned in \S \ref{section:3_5}, it follows from the arguments in \S \ref{section:3_3} that 
$2 \in Z(Y, X, h; f_Y)$. 
Take a system $\{(V_j, e_j, w_j)\}$ of type $2$ such that the corresponding local weight function $\varphi_j$ satisfies $\varphi_j(z_j, w_j)=O(|w_j|^2)$ as $|w_j|\to 0$ for each $j$. 
Consider the set $S$ of all positive integers $m$ such that 
\[
\left.\frac{\partial^{2m} \varphi_j}{\partial w_j^m\partial\overline{w_j}^m}\right|_{U_j}\not\equiv 0. 
\]

First, let us consider the case where $S=\emptyset$. 
In this case, by Lemma \ref{lem:for_thm_main_2} below, one have that there exists a pluriharmonic function $\eta_j^{(n)}$ for any positive integer $n$ such that $\varphi_j=\eta_j^{(n)}+O(|w_j|^n)$ as $|w_j|\to 0$. 
%In this case, we have that $t_{jk}w_k=w_j$ holds on a neighborhood of $U_{jk}$ for each $j$ and $k$ ($t_{jk}$ is a unitary constant as in \S \ref{section:3}). 
%Indeed, for the expansion 
%\[
%t_{jk}w_k = w_j + \sum_{n=2}^\infty f_{kj, n}(z_j)\cdot w_j^n
%\]
%on a neighborhood of $U_{jk}$, one can inductively show that $f_{kj, n}\equiv 0$ as follows: 
%First, one have that $f_{kj, 2}\equiv 0$ since $\{(V_j, e_j, w_j)\}$ is a system of type $2$. 
%Next, if $f_{kj, \nu}\equiv 0$ holds for any $\nu$ with $\nu<n$, one can run the same argument as in the proof of Proposition \ref{prop:higher_main} $(ii)$ and $(iii)$ to show that $f_{kj, n}\equiv 0$. 
Therefore, when $S=\emptyset$, the assertion $(iii)$ holds if $\Theta_h\wedge \Theta_h\equiv 0$. When there exists a point of $X$ at which $\Theta_h\wedge \Theta_h\not= 0$, 
it follows from \cite[Proposition 2]{B} that the assertion $(i)$ holds (see Remark \ref{rmk:1conv_V-Y} below for details). 
%Note that the plurisubharmonic exhaustion function $\psi_Y:=-\log |f_Y|_h^2$ only depends on $|w_j|$'s, since each equation $|w_j|=\text{constant}$ defines a compact Levi-flat hypersurface whose leaves are dense. 

Next, consider the case where $S\not=\emptyset$. 
Denote by $m$ the minimum of $S$. 
Take a holomorphic function $f_j^{(m-1)}$ as in Lemma \ref{lem:for_thm_main_2}. 
Then the expansion of $\varphi_j$ is in the form
\begin{align*}
\varphi_j(z_j, w_j) = &f_j^{(m-1)}(z_j, w_j) +\overline{f_j^{(m-1)}(z_j, w_j)}\\
&+\sum_{p, q\geq 0, p+q= 2m}\varphi_j^{(p, q)}\cdot w_j^p\overline{w_j}^q
 %+ \varphi_j^{(2m, 0)}(z_j)\cdot w_j^{2m}
% + \varphi_j^{(2m, 0)}(z_j)\cdot \overline{w_j}^{2m}
 + O(|w_j|^{2m+1}). 
\end{align*}
Define a new system $\{(V_j, \widehat{e}_j, \widehat{w}_j)\}$ by $\widehat{e}_j:=e_j\cdot \exp(f_j^{(m-1)})$ and $\widehat{w}_j:=w_j\cdot \exp(-f_j^{(m-1)})$. 
Then one have that the corresponding (new) local weight function $\widehat{\varphi}_j$ satisfies 
\begin{align*}
\widehat{\varphi}_j &= \varphi_j -\log\left|\exp (f_j^{(m-1)}) \right|^2 \\
&= \sum_{p, q\geq 0, p+q= 2m}\varphi_j^{(p, q)}\cdot w_j^p\overline{w_j}^q + O(|w_j|^{2m+1}) \\
&= \sum_{p, q\geq 0, p+q= 2m}\varphi_j^{(p, q)}\cdot \widehat{w}_j^p\overline{\widehat{w}_j}^q + O(|\widehat{w}_j|^{2m+1}). 
\end{align*}
Therefore, by replacing our system $\{(V_j, e_j, w_j)\}$ with $\{(V_j, \widehat{e}_j, \widehat{w}_j)\}$, one may assume that
$\{(V_j, w_j)\}$ is a system of type $2m-1$ and that the expansion of the local weight function $\varphi_j$ with respect to the local frame $e_j$ is in the form 
\[
\varphi_j(z_j, w_j) = \sum_{p, q> 0, p+q= 2m}A_j^{(p, q)}\cdot w_j^p\overline{w_j}^q
 + \varphi_j^{(2m, 0)}(z_j)\cdot w_j^{2m}
 + \varphi_j^{(2m, 0)}(z_j)\cdot \overline{w_j}^{2m}
 + O(|w_j|^{2m+1}). 
\]
By Lemma \ref{lem:higher_inductive} and the assumption that $N_{Y/X}$ is non-torsion, 
one have that $A_j^{(p, q)}\equiv 0$ if $p\not=q$. 
Again by Lemma \ref{lem:higher_inductive} and our definition of $m$, one have that there exists a positive constant $A$ such that $A_j^{(m, m)}\equiv A$. 
Therefore, by Lemma \ref{lem:hol_higher} and an argument in the proof of Proposition \ref{prop:higher_main} $(i)$, we can rewrite the above expansion into the form
\[
\varphi_j(z_j, w_j) = A\cdot |w_j|^{2m} + O(|w_j|^{2m+1}). 
\]
by changing the system $\{(V_j, w_j)\}$ again if necessary. 
Thus one can calculate the complex Hessian as follows: 
\[
\left(
    \begin{array}{cc}
      (\varphi_j)_{w_j \overline{w_j}} & (\varphi_j)_{w_j \overline{z_j}} \\
      (\varphi_j)_{z_j \overline{w_j}} & (\varphi_j)_{z_j \overline{z_j}}
    \end{array}
  \right)
=\left(
    \begin{array}{cc}
      m^2A|w_j|^{2m-2} + O(|w_j|^{2m-1}) & O(|w_j|^{2m}) \\
      O(|w_j|^{2m}) & O(|w_j|^{2m+1})
    \end{array}
  \right). 
\]
From this calculation, one have that there exists a neighborhood $V$ of $Y$ such that the curvature $\Theta_h$ has at least one positive eigenvalue at each point of $V\setminus Y$. 

In the case where there exists a point in $V\setminus Y$ at which $\Theta_h$ has two positive eigenvalues, it follows from \cite[Proposition 2]{B} that the assertion $(i)$ holds (see Remark \ref{rmk:1conv_V-Y} for details). 
Thus, in what follows, we may assume that the rank of the complex Hessian of each $\varphi_j$ is one at each point of $V\setminus Y$. 
Then it follows from \cite{S} and \cite[Theorem 2.4]{BK} that there exists a non-singular holomorphic (Monge--Amp\`ere) foliation $\mathcal{G}$ on $V\setminus Y$. 
Let $G\subset T_{V\setminus Y}$ be the corresponding subbundle: $G:=T_\mathcal{G}$. 
Denote by $s\colon V\setminus Y \to {\bf P}(T_{V\setminus Y})$ the section whose image coincides with ${\bf P}(G)$, where we denote by ${\bf P}$ the relative projectivization. 

For each point $p\in V_j\cap (V\setminus Y)$, take complex numbers $a_j(p)$ and $b_j(p)$ such that  
\[
s(p) = \left[ a_j(p)\cdot \frac{\partial}{\partial w_j} + b_j(p)\cdot \frac{\partial}{\partial z_j}\right]. 
\]
Note that the ratio $a_j/b_j$ defines a meromorphic function on $V_j\cap (V\setminus Y)$. 
By the definition of Monge--Amp\`ere foliation, one have that 
\[
\left(
    \begin{array}{cc}
      (\varphi_j)_{w_j \overline{w_j}}(p) & (\varphi_j)_{w_j \overline{z_j}}(p) \\
      (\varphi_j)_{z_j \overline{w_j}}(p) & (\varphi_j)_{z_j \overline{z_j}}(p)
    \end{array}
  \right)
\left(
    \begin{array}{c}
      a_j(p) \\
	b_j(p)
    \end{array}
  \right)
=\left(
    \begin{array}{c}
      0 \\
      0
    \end{array}
  \right). 
\]
holds at each point of $V_j\cap (V\setminus Y)$. 
Therefore, by a simple argument, one have the estimate 
$a_j(p)/b_j(p) = O(|w_j(p)|)$ as $p$ approaches to $U_j$. 
From Riemann's extension theorem and this estimate, it follows that the section $\widetilde{s}\colon V\to {\bf P}(T_V)$ defined by 
\[
\widetilde{s}(p):=\begin{cases}
\left[\frac{\partial}{\partial z_j}\right] & \text{if}\ p\in Y\\
s(p) & \text{if}\ p\in V\setminus Y
\end{cases}
\]
is also holomorphic. 

Denote by $F\subset T_V$ the subbundle which corresponds to the image of $\widetilde{s}$. 
This subbundle $F$ is clearly integrable, since the integrability condition is trivial for a rank-one subbundle of the tangent bundle of surfaces, from which the theorem follows. 
\qed

\begin{lemma}\label{lem:for_thm_main_2}
Let $X, Y, L$, and $h$ be as in Theorem \ref{thm:main_2}. 
Assume that $N_{Y/X}$ is non-torsion in the Picard variety of $Y$. 
Take a system $\{(V_j, e_j, w_j)\}$ of type $2$ such that the corresponding local weight function $\varphi_j$ satisfies $\varphi_j(z_j, w_j)=O(|w_j|^2)$ as $|w_j|\to 0$ for each $j$, 
and an integer $m\geq 0$ which is less than any element of $S$, where $S$ is the set as in the proof of Theorem \ref{thm:main_2}. 
Then we have 
\[
\varphi_j(z_j, w_j) = f_j^{(m)}(z_j, w_j) +\overline{f_j^{(m)}(z_j, w_j)} + O(|w_j^{2m+2}|), 
\]
where $f_j^{(m)}$ is a holomorphic function on $V_j$ in the form 
\[
f_j^{(m)}(z_j, w_j) = \sum_{\nu=1}^{2m+1} \varphi_j^{(\nu, 0)}(z_j)\cdot w_j^\nu. 
\]
Here $\varphi_j^{(\nu, 0)}$ is a holomorphic function on $U_j$ for each $\nu$ with $2\leq \nu\leq 2m$. 
\end{lemma}

\begin{proof}
The assertion clearly holds for $m=0$. 
Assume the assertion for an integer $\mu$ with $0\leq \mu <m$ as the inductive assumption. 
Assume also that $m$ is less than any element of $S$. 
%Let $u:=(1, 0)$ and $v:=(0, 1)$ be basis of $\mathbb{C}^2$. 
Then, it follows by applying Lemma \ref{lem_positive_function} for $(H_j)^0_0$ that 
the expansion of $\varphi_j$ is in the form 
\begin{align*}
\varphi_j = &f_j^{(m-1)}(z_j, w_j) +\overline{f_j^{(m-1)}(z_j, w_j)} 
+ \varphi_j^{(2m, 0)}(z_j)\cdot w_j^{2m}
+\overline{\varphi_j^{(2m, 0)}(z_j)}\cdot \overline{w_j}^{2m}\\
&+ \varphi_j^{(2m+1, 0)}(z_j)\cdot w_j^{2m+1} 
+ \overline{\varphi_j^{(2m+1, 0)}(z_j)}\cdot \overline{w_j}^{2m+1} + O(|w_j|^{2m+2}), 
\end{align*}
where $H_j$ is the $2\times 2$ matrix with entries $(H_j)^a_b$ with $0\leq a, b\leq 1$ defined by 
\[
(H_j)^a_b = \begin{cases}
 \frac{\partial^2 \varphi_j}{\partial w_j \partial \overline{w_j}} & \text{if}\ a=b=0 \\
 \frac{\partial^2 \varphi_j}{\partial w_j \partial \overline{z_j}} & \text{if}\ a=0, b=1 \\
 \frac{\partial^2 \varphi_j}{\partial z_j \partial \overline{w_j}} & \text{if}\ a=1, b=0 \\
 \frac{\partial^2 \varphi_j}{\partial z_j \partial \overline{z_j}} & \text{if}\ a=b=1
\end{cases}
\]
i.e. $H_j$ is the complex Hessian of $\varphi_j$. 
As it holds that ${\rm det}\,H_j=O(|w_j|^{4m-2})$ and the coefficient of $|w_j|^{4m-2}$ of the expansion of this function is $-4m^2|(\varphi_j^{(2m, 0)})_{\overline{z_j}}|^2$, 
It follows again from Lemma \ref{lem_positive_function} that $\varphi_j^{(2m, 0)}$ is holomorphic. 
By a similar argument, one also have that $\varphi_j^{(2m+1, 0)}$ is holomorphic. 
\end{proof}

\begin{remark}\label{rmk:1conv_V-Y}
In order to conclude that $X\setminus Y$ is holomorphically convex by using \cite[Proposition 2]{B}, as $X\setminus Y$ is a connected K\"ahler surface with trivial canonical bundle, 
one need to show the existence of a $C^\infty$ plurisubharmonic exhaustion function $f\colon X\setminus Y\to \mathbb{R}$ which is strictly plurisubharmonic at some point. 
To show this, we use the existence of a point $p\in X\setminus Y$ at which the curvature $\Theta_h$ has two positive eigenvalues. 
Let $\psi$ be a $C^\infty$ plurisubharmonic exhaustion function of $X$ (we need to assume that $X$ is weakly $1$-complete in order to assure the existence of this). 
Then the function $f:=\psi-\log |f_Y|_h^2$ enjoys the condition above, 
where $f_Y$ is the canonical section of $[Y]$. 
\end{remark}

%%%%%%%%%%%%
\subsection{Proof of Theorem {\ref{thm:main_linearization}}}

In this subsection, we prove the following: 

\begin{theorem}\label{thm:linearizability_genaral}
Let $Y$ be a compact non-singular hypersurface of a complex manifold $X$ 
and $L$ be a holomorphic line bundle on $X$. 
Assume that the restriction $L|_Y$ of $L$ to $Y$ is unitary flat, 
the line bundle $K_Y^{-1}\otimes N_{Y/X}^{-1}$ is semi-positive, 
and that, for any neighborhood $V$ of $Y$ in $X$, there exists a $1$-convex neighborhood of $Y$ in $V$ whose maximal compact analytic set is $Y$. 
Then the restriction of $L$ to a neighborhood is unitary flat if $u_1(Y, X, L)=0$. 
\end{theorem}

\begin{proof}
Take a open cover $\{U_j\}$ of $Y$, small open subset $V_j$ of $X$ such that $U_j$, 
local defining function $w_j$ of $U_j$ in $V_j$, and local frame $e_j$ of $L$ on $V_j$ as in \S \ref{section:3}. 
Assume that $u_1(Y, X, L)=0$. 
Then, from the same argument as in the proof of Theorem \ref{thm:1}, 
it follows that one may assume that the system $\{(V_j, e_j, w_j)\}$ is of type $2$ by modifying them if necessary; namely there exists a constant $t_{jk}\in\mathrm{U}(1)$ such that 
$t_{jk}^{-1}e_k/e_j=1+O(w_j^2)$ on each $V_{jk}$. 
By shrinking $V_j$'s, we may also assume that $|t_{jk}^{-1}e_k/e_j-1|<1$ on each $V_{jk}$. 
Define a function $a_{jk}\colon V_{jk}\to \mathbb{C}$ by 
\[
a_{jk}:=\log\left(\frac{t_{jk}^{-1}e_k}{e_j}\right), 
\]
where we are using the branch such that $\log 1=0$. 
It can easily be observed that $a_{jk}=O(w_j^2)$ and that $a_{jk}+a_{k\ell}+a_{\ell j}=0$ holds on each $V_j\cap V_k\cap V_\ell$. 
Therefore, $\{(V_{jk}, a_{jk})\}$ defines an element of \v{C}ech cohomology group $\check{H}^1(\{V_j\}, \mathcal{O}_V(-2Y))$. 

By shrinking if necessary, we will assume that $V:=\bigcup_jV_j$ is a $1$-convex neighborhood of $Y$ whose maximal compact analytic set is $Y$ in what follows. 
Then, according to Ohsawa's vanishing theorem \cite[Theorem 4.5]{O}, it follows that 
$H^1(V, [Y]^{-2})=0$, since $([Y]^{-2}\otimes K_V^{-1})|_Y\cong N_{Y/X}^{-2}\otimes (N_{Y/X}^{-1}\otimes K_Y)^{-1}=N_{Y/X}^{-1}\otimes K_Y^{-1}$ is semi-positive. 
Therefore one have that the \v{C}ech cohomology class defined by $\{(V_{jk}, a_{jk})\}$ is trivial, which means that there exists a holomorphic function $b_j\colon V_j\to \mathbb{C}$ on each $V_j$ such that $b_j=O(w_j^2)$ and that $-b_j+b_k=a_{jk}$ holds on each $V_{jk}$. 

Let $\widehat{e}_j$ be a new frame defined by $\widehat{e}_j:=e_j\cdot \exp(-b_j)$. 
Then it follows from a simple calculation that $t_{jk}^{-1}\widehat{e}_k=\widehat{e}_j$, which prove the theorem. 
\end{proof}

\begin{proof}[Proof of Theorem {\ref{thm:main_linearization}}]
By Theorem \ref{thm:1}, it is sufficient to show that $L$ is semi-positive by assuming that $u_1(Y, X, L)=0$. 
As ${\rm deg}\,N_{Y/X} <0$, $Y$ admits a fundamental system of neighborhoods which consists of $1$-convex neighborhoods of $Y$ whose maximal compact analytic sets are $Y$. 
The degree of the line bundle $K_Y^{-1}\otimes N_{Y/X}^{-1}$ is non-negative, since ${\rm deg}\,N_{Y/X}\leq 2-2g$, from which it follows that $K_Y^{-1}\otimes N_{Y/X}^{-1}$ is semi-positive. 
Therefore it follows from Theorem \ref{thm:linearizability_genaral} that there exists a neighborhood $V$ of $Y$ such that the restriction $L|_V$ is unitary flat. 

Take a $C^\infty$ Hermitian metric $h_0$ on $L\otimes [Y]^{-m}$ with semi-positive curvature. 
Let $f_Y$ be the canonical section of $[Y]$. 
Denote by $h_Y$ the singular Hermitian metric on $[Y]$ defined by $|f_Y|_{h_Y}^2\equiv 1$. 
Then a $C^\infty$ Hermitian metric on $L=(L\otimes [Y]^{-m})\otimes [Y]^m$ can be constructed from the flat metric on $L|_V$ and the singular Hermitian metric $h_0\cdot h_Y^m$ on $L$ 
by using ``regularized minimum construction", see \cite{K2}, \cite[\S 5]{K2018}, or \cite[\S 2.1]{K2019} for the detail.

\end{proof}
%%%%%%%%%%%%

%%%%%%%%%%%%%%%%%%%%%%%%%%%%%%%%%%%%%%%%%%%

%%%%%%%%%%%%%%%%%%%%%%%%%%%%%%%%%%%%%%%%%%%
\section{Problems}\label{section:5}

Towards solving Conjecture \ref{conj:main} or \cite[Conjecture 2.1]{K2019}, here we make some discussion. 

First, consider one of the most simplest cases: when $Y$ is a non-singular compact curve holomorphically embedded into a non-singular surface $X$ such that the normal bundle is topologically trivial. 
Assume that the line bundle $[Y]$ is semi-positive. 
Take a $C^\infty$ Hermitian metric $h$ with semi-positive curvature. 
As it follows from \cite[Theorem 1.1]{K3} that the pair $(Y, X)$ is of infinite type, 
we are interested in determining whether or not there exists such an example of $(Y, X)$ of type $(\gamma)$ in Ueda's classification \cite[\S 5]{U}. 
From the viewpoint of Theorem \ref{thm:main_2} and the study on the relation between the holonomy of the foliation and Ueda type we investigated in \cite{KO}, we are interested in the following question at least when $h$ is real analytic: 

\begin{question}\label{q:surface_main}
Let $Y$ be a smooth elliptic curve holomorphically embedded into a non-singular K\"ahler surface $S$. 
Assume that the canonical bundle $K_S$ or the anti-canonical bundle $K_S^{-1}$ is semi-positive. 
Assume also that there exists an integer $m$ and a divisor $D$ of $S$ such that $K_S=[mY+D]$ and that the support of $D$ does not intersect $Y$. 
Let $X$ be a sufficiently small weakly $1$-complete neighborhood of $Y$. 
\\
$(i)$ Is there an example of (Y, X) of type $(\gamma)$ such that $X\setminus Y$ is holomorphically convex? \\
$(ii)$ How is the holonomy of the foliation $\mathcal{F}$ as in Theorem \ref{thm:main_2} $(ii)$ have (if exists)? 
\end{question}
Note that many things are known for Question \ref{q:surface_main} when $S$ is projective and $K_S$ is semi-positive, since abundance conjecture is affirmative for projective surfaces. 

Next, let us consider the case of general dimensions. 
According to \cite[Theorem 1.1]{K3} and Theorem \ref{thm:flat_u2}, 
it seems to be natural to pose the following: 

\begin{conjecture}
Let $X$ be a complex manifold and $Y$ be a non-singular compact hypersurface of $X$ which is K\"ahler. 
Assume that the normal bundle $N_{Y/X}$ is topologically trivial, 
and that $[Y]$ is semi-positive. 
Then the pair $(Y, X)$ is of infinite type. 
\end{conjecture}

We are also interested in the case where $Y$ is a hypersurface and $N_{Y/X}^{-1}$ is not pseudo-effective, since in this case, by Proposition \ref{prop:well-def_main} $(iii)$, the well-definedness of the $n$-th obstruction classes are assured for any $n$ whenever there exists a system of type $n$. 
By Theorem \ref{thm:2}, it seems to natural to ask the following: 

\begin{problem}
Let $X$ be a complex manifold and $Y$ be a non-singular compact hypersurface of $X$ which is K\"ahler. 
Let $L$ be a semi-positive line bundle on $X$ such that $L|_Y$ is topologically trivial. 
Assume that the conormal bundle $N_{Y/X}^{-1}$ is not pseudo-effective. 
Then, does it hold that $u_n(Y, X, L)=0$ for any positive integer $n$? 
\end{problem}

%%%%%%%%%%%%%%%%%%%%%%%%%%%%%%%%%%%%%%%%%%%

%%%%%%%%%%%% References %%%%%%%%%%%%%

\end{document}